\documentclass[11pt,reqno]{amsart}
\usepackage{a4wide}

\usepackage{graphicx}
\usepackage{enumerate}
\usepackage{hyperref}
\usepackage{cleveref}
\usepackage{amsmath,amsopn,amssymb,amsfonts,stmaryrd}
\usepackage{verbatim}
\usepackage{amsthm}
\usepackage{mathtools}
\usepackage{color}
\usepackage{enumitem}
\usepackage[framemethod=TikZ]{mdframed}
\usepackage{bbm}
\usepackage{mathrsfs}
\usepackage{booktabs}
\usepackage{caption}
\usepackage{bm}
\usepackage{tensor}

\makeatletter
\@namedef{subjclassname@2020}{%
	\textup{2020} Mathematics Subject Classification}
\makeatother

\newcommand{\IR}{{\mathbb R}}%Reals
\newcommand{\CC}{\mathcal{C}}

\newcommand{\C}{\mathbb C}

\newcommand\restr[2]{\ensuremath{\left.#1\right|_{#2}}}

\theoremstyle{plain}
\newtheorem{thm}{Theorem}[section]
\newtheorem{cor}[thm]{Corollary}
\newtheorem{lem}[thm]{Lemma}
\newtheorem{prop}[thm]{Proposition}

\newtheorem{rem}[thm]{Remark}

\theoremstyle{definition}

\newtheorem{defn}[thm]{Definition}

\numberwithin{equation}{section}

\def\d{\delta}

\def\s{\sigma}

\def\d{\delta}

\def\s{\sigma}

\def\QQ{{\mathbb Q}}
\def\CC{{\mathbb C}}
\def\RR{{\mathbb R}}
\def\NN{{\mathbb N}}
\def\ZZ{{\mathbb Z}}
\def\d{{\mathrm{d}}}

\def\Sd{{\mathbb{S}^3}}
\def\Spv{{\mathbb{S}^3(p,v)}}

\def\SS{{\mathbb{S}}}
\def\S{{\mathbb{S}^{2n+1}}}

\def\i{{\mathrm{i}}}

\def\Re{{\mathrm{Re}}}

\def\II{\mathrm{\RNum{2}}}

\newcommand{\gc}{\gamma}
\newcommand{\dgc}{\dot{\gc}}
\newcommand{\ddgc}{\ddot{\gc}}

\def\bar{\overline}

\setlist[itemize]{noitemsep, topsep=0pt}

\makeatletter
\newcommand{\vast}{\bBigg@{2}}
\newcommand{\Vast}{\bBigg@{5}}
\makeatother

\newcommand{\RNum}[1]{\uppercase\expandafter{\romannumeral #1\relax}}
\allowdisplaybreaks

\title[Magnetic geodesics on odd-dimensional spheres]{The Hopf--Rinow theorem and the Mañé critical value\\ for magnetic geodesics on odd-dimensional spheres} 
\author{P. Albers}
\address{Faculty of Mathematics and Computer Science,
	University of Heidelberg,
	Im Neuenheimer Field 205,
	69120 Heidelberg, Germany}
\email{palbers@mathi.uni-heidelberg.de}

\author{G. Benedetti}
\address{Department of Mathematics,
	Vrije Universiteit Amsterdam - Faculty of Science,
	De Boelelaan 1111,
	1081 HV Amsterdam, The Netherlands}
\email{g.benedetti@vu.nl}

\author{L. Maier}
\address{Faculty of Mathematics and Computer Science,
	University of Heidelberg,
	Im Neuenheimer Field 205,
	69120 Heidelberg, Germany}
\email{lmaier@mathi.uni-heidelberg.de}



\begin{document}
	\maketitle
	\renewcommand{\abstractname}{Abstract}
	\begin{abstract}
		The subject of this article are magnetic geodesics on odd-dimensional spheres endowed with the round metric and with the magnetic potential given by the standard contact form. We compute the Mañé critical value of the system and show that a value of the energy is supercritical if and only if all pairs of points on the sphere can be connected by a magnetic geodesic with that value of the energy. Our methods are explicit and rely on the description of the submanifolds invariant by the flow and of the symmetries of the system, which we define for a general magnetic system and call totally magnetic submanifolds and magnetomorphisms, respectively. We recover hereby the known fact that the system is super-integrable: the three-spheres obtained intersecting the ambient space with a complex plane are totally magnetic and each magnetic geodesic is tangent to a two-dimensional Clifford torus. In our study the integral of motion given by the angle between magnetic geodesics and the Reeb vector field plays a special role, and can be used to realize the magnetic flow as an interpolation between the sub-Riemannian geodesic flow of the contact distribution and the Reeb flow of the contact form. 
	\end{abstract}

	\section{Introduction and statement of results}
	In the 1960s the motion of a charged particle in a magnetic field was put into the context of modern dynamical systems by V.\ Arnold in his pioneering work \cite{Arnold2009}. The motion has the following mathematical description. Let $(M,g)$ be a closed, connected Riemannian manifold and $\sigma\in\Omega^2(M)$ be a closed two-form. The form $\sigma$ is called \emph{magnetic field} and the triple $(M,g,\sigma)$ is called \emph{magnetic system}. This determines the skew-symmetric bundle endomorphism $Y\colon TM\to TM$, the \emph{Lorentz force}, by
	\begin{equation}\label{e:Lorentz}
	g_q\left(Y_qu,v\right)=\sigma_q(u,v),\qquad \forall\, q\in M,\ \forall\,u,v\in T_qM.
	\end{equation}
Then a smooth curve $\gc\colon \RR\to M$  satisfying \begin{equation}\label{e:mg}
		\nabla_{\dgc}\dgc= Y_{\gc}\dgc
	\end{equation}
	is called a \emph{magnetic geodesic} of $(M,g,\sigma)$. Here $\nabla$ denotes the Levi-Civita of the metric $g$. From \eqref{e:mg}, we see that a magnetic geodesic with $\sigma=0$ is a standard geodesic for the metric $g$ and we can consider \eqref{e:mg} as a linear perturbation of the geodesic equation. Therefore, one of the main points of interest is to work out the similarities and differences between standard and magnetic geodesics.
	
Since $Y$ is skew-symmetric, magnetic geodesics have constant kinetic energy $E(\gamma,\dot\gamma):=\tfrac12g_\gamma(\dot\gamma,\dot\gamma)$, and hence constant speed $|\dot\gamma|:=\sqrt{g_\gamma(\dot\gamma,\dot\gamma)}$, just like standard geodesics.
	
Energy conservation is a footprint of the Hamiltonian nature of the system. Indeed, let us define the \emph{magnetic geodesic flow} on the tangent bundle by
	\[
	\varPhi_{g,\sigma}^t\colon TM\to TM,\quad (q,v)\mapsto \left( \gc_{q,v}(t),\dgc_{q,v}(t)\right),\quad \forall t\in\RR,
	\] where $\gc_{q,v}$ is the unique magnetic geodesic with initial value $(q,v)\in TM$. By \cite{Gin}, $\Phi^t_{g,\sigma}$ is the Hamiltonian flow given by the kinetic energy $E\colon TM\to\RR$ and the twisted symplectic form 
	\[
	\omega_\sigma:=\d\lambda-\pi^*_{TM}\sigma,
	\]
	where $\lambda$ is the metric pullback of the canonical Liouville $1$-form from $T^*M$ to $TM$ and $\pi_{TM}\colon TM\to M$ is the projection.
	
	However, differently from the case of standard geodesics, magnetic geodesics of different speeds are not just reparametrization of unit speed magnetic geodesics. For instance, if $M=\SS^2$, $g$ has constant curvature $1$ and $\sigma$ is the corresponding area form, then magnetic geodesics of kinetic energy $k$ are geodesic circles of radius $\arctan(\sqrt
	{2k})$ traversed counterclockwise, see \cite{BKmag}. What we can say in general is that, for every $s>0$, $\gamma$ is a magnetic geodesic of $(M,g,\sigma)$ with speed $\tfrac1s$ if and only if the unit speed reparametrization of $\gamma$ is a magnetic geodesic of $(M,g,s\sigma)$. This means that investigating magnetic geodesics of $(M,g,\sigma)$ for different speeds is the same as investigating the magnetic geodesics of $(M,g,s\sigma)$ with unit speed and varying $s>0$. This approach has the advantage that for $s=0$ we recover standard geodesics and for $s<0$, we recover the magnetic geodesics of $(M,g,\sigma)$ with the opposite orientation. With this notation, $s$ will be called the \emph{strength} of the magnetic geodesic. 
	\medskip
	
	To prove their landmark result about the equivalence between geodesic completeness and metric completeness, Hopf and Rinow showed that every two points on a connected and geodesically complete manifold can be connected by a standard geodesic, see \cite{HR,Milnor}. Since closed manifolds are geodesically complete, it follows that every two points on $M$ are connected by a standard geodesic. As seen above, this is independent of the speed (or energy) of the required geodesic. In the present article we investigate whether this result remains true for magnetic geodesics with a prescribed kinetic energy. In other words, we ask for which points $p,q\in M$ and for which $k>0$ there is a magnetic geodesic with energy $k$ connecting $p$ to $q$.
	
	 An important role in this question is played by the Mañé critical value $c(M,g,\sigma)$ of the universal cover $\hat M$ of $M$. If $\sigma$ is weakly exact, that is, the lift $\hat\sigma$ of $\sigma$ is exact, we define
	\begin{equation}\label{e:maneuni}
   c(M,g,\sigma):=\inf_{\d\hat\alpha=\hat\sigma}\tfrac12\Vert \hat\alpha\Vert_\infty^2\;\in[0,\infty],
	\end{equation}
where the infimum is taken over all primitives $\hat\alpha$ of $\hat\sigma$ and $\Vert\cdot\Vert_\infty$ denotes the supremum norm over $\hat M$ with respect to the lift $\hat g$ of the metric $g$. Notice that $c(M,g,\sigma)<\infty$ if and only if $\hat\sigma$ has a bounded primitive $\hat\alpha$, that is, $\Vert\hat\alpha\Vert_\infty<\infty$.

The Mañé critical value has equivalent definitions in terms of the action functional. In particular, if $M$ is simply connected and $\sigma$ is exact, then $-c(M,g,\sigma)$ is the minimal action of a Borel invariant probability measure for the system. As a consequence, the energy level $\{E=c(M,g,\sigma)\}\subset TM$ contains the so-called Mather set $\widetilde{\mathcal M}$ which is the support of the action-minimizing probability measures of the system \cite{Sor}. We refer to Section \ref{Section manes value} for more background on the Mañé's critical value.\medskip
	
	Our interest in the critical value relies on the fact that if $k>c(M,g,\sigma)$ every two points on $M$ can be connected by a magnetic geodesic of energy $k$, as follows from \cite[Corollary B]{Co06} when $\sigma$ is exact and, more generally, from \cite[Theorem 3.2]{Me09} when $\sigma$ is weakly exact. Indeed, if $k>c(M,g,\sigma)$, there exists a primitive $\hat\alpha$ such that \begin{equation}\label{e:finsler}
	F\colon T\hat M\to\RR,\qquad F(\hat q,\hat v):=\sqrt{2k}|\hat v|_{\hat q}-\hat\alpha_{\hat q}(\hat v),\quad \forall\,(\hat q,\hat v)\in T\hat M
	\end{equation}
	is a Finsler metric on $\hat M$ whose geodesics are lifts of magnetic geodesics on $M$ with energy $k$. Finsler geodesics connecting pairs of points on $\hat M$ can then be found by minimizing the $F$-length of paths and project to the desired magnetic geodesics on $M$, and such magnetic geodesics are close to standard geodesics in a sense made precise in \cite{PS}.
	
	If $k\leq c(M,g,\sigma)$, then there might be pairs of points which cannot be not connected by a magnetic geodesic with energy $k$, as the example of $(\SS^2,g,\sigma)$ mentioned above shows. Indeed, in this case $c(\SS^2,g,\sigma)=\infty$ since $\SS^2$ is simply connected and $\sigma$ is not exact, and two points can be connected by a magnetic geodesic of energy $k$ if and only if they are at distance at most $2\arctan(\sqrt{2k})$, which is strictly less than $\pi$, the diameter of $\SS^2$. In higher dimensions, except for the case of Kähler magnetic systems with constant holomorphic curvature \cite{Ada95}, there is no good understanding about which pairs of points fail to be connected for a given subcritical energy, even for simple systems. 
	
	The main result of this article is to analyze the existence of magnetic geodesics connecting pairs of points for the magnetic system 
	\begin{equation}\label{e:magsys}
	(\S,g,\d\alpha).
	\end{equation}
	Here the manifold $\S$ is the sphere of radius $1$ in $\C^{n+1}$ with standard Hermitian product $\langle\cdot,\cdot\rangle$, the metric  $g=\mathrm{Re}\langle\cdot,\cdot\rangle$ is the restriction of the Euclidean metric to $\S$, and the magnetic potential $\alpha$ is the standard contact form on $\S$, that is, $\alpha_z=\tfrac12\mathrm{Re}\langle \mathrm iz,\cdot\rangle$ for all $z\in\S$. The Reeb vector field of $\alpha$ is the unique vector field $R$ on $\S$ such that $\d\alpha(R,\cdot)=0$ and $\alpha(R)=1$. In this case, we get $R_z=2\mathrm iz$. The trajectories of the flow 
	\begin{equation}\label{e:hopfflow}
	\Phi_R^t(z)=e^{2\mathrm it}z,\qquad \forall\, t\in\RR,\ z\in\S
	\end{equation}
	are the fibers of the Hopf map $\pi\colon \S\to \mathbb CP^n$, which sends each point on $\S$ to the complex line through it. This is the simplest example of a Zoll Reeb flow, where all orbits of the Reeb vector field are periodic and with the same minimal period \cite{APB,ABE24}. Finally, $\ker\alpha$ is the contact distribution of $\alpha$ and, in our case, coincides with the orthogonal of $R$.  
	\begin{thm}\label{t:mane}
	The Mañé's critical value of the system is
	\[
	c(\S,g,\d\alpha)=\tfrac{1}{2}\Vert\alpha\Vert_\infty^2=\tfrac18.
	\]
	The Mather set of the system is 
	\[
	\widetilde{\mathcal M}=\{(z,\tfrac14 R_z)\mid z\in\S\}\subset T\S.
	\]
	Let $q_0$ and $q_1$ be two points on $\S$ and denote by $\langle q_0,q_1\rangle$ their Hermitian product. For every $k>0$, let $\mathcal G_k(q_0,q_1)$ be the set of magnetic geodesics with energy $k$ connecting $q_0$ and $q_1$. We have the following three cases
	\begin{enumerate}
		\item if $k>\tfrac18$, then $\mathcal G_k(q_0,q_1)\neq\varnothing$;
		\item if $k=\tfrac18$, then $\mathcal G_k(q_0,q_1)\neq\varnothing$ if and only if $\langle q_0,q_1\rangle\neq0$;
		\item if $k<\tfrac18$, we have the following three subcases 
		\begin{enumerate}
			\item if $|\langle q_0,q_1\rangle|> \sqrt{1-8k}$, then $\mathcal G_k(q_0,q_1)\neq\varnothing$;
			\item if $|\langle q_0,q_1\rangle|=\sqrt{1-8k}$, then there are $a_k,b_k\in\RR$ with $b_k>0$ such that $\mathcal G_k(q_0,q_1)\neq\varnothing$ if and only if $\langle q_0,q_1\rangle=e^{\mathrm i(a_k+mb_k)}\sqrt{1-8k}$ for some $m\in\ZZ$;
			\item if $|\langle q_0,q_1\rangle|< \sqrt{1-8k}$, then $\mathcal G_k(q_0,q_1)=\varnothing$.
		\end{enumerate}
	\end{enumerate}
	\end{thm}
	\begin{rem}
		The result about the Mañé's critical value and the Mather set generalizes to magnetic systems $(M,g,\d\alpha)$, where $M$ is simply connected and the metric dual vector field $X$ of $\alpha$ is a Killing vector field for $g$, see Theorem \ref{t:manegeneral}
	\end{rem}
	\begin{rem}
	The fact that $\mathcal G_k(q_0,q_1)\neq\varnothing$ for $k>\tfrac18$ also follows by Contreras' result \cite[Corollary B]{Co06}, which constructs an action minimizing magnetic geodesics, or equivalently, a length-minimizing geodesic for a Finsler metric like in \eqref{e:finsler}. On the contrary our method of proof of (1), (2), (3) in Theorem \ref{t:mane} is not variational but relies on an explicit topological argument, see Subsection \ref{ss:mane2}. However, for $n=1$ and upon fixing $q_0$, we also show that for every $k<\tfrac18$, the Mañé critical value of the system restricted to the set $B_{q_0,k}:=\{q_1\in\S\mid |\langle q_0,q_1\rangle|> \sqrt{1-8k}\}$ is at least $k$, see Subsection \ref{ss:manehopf}. More concretely, we construct a primitive $\alpha_k$ of $\d\alpha$ on $B_{q_0,k}$ such that $\tfrac12\Vert \alpha_k\Vert_\infty^2<k$. A variational proof of (3) following this observation would require an extra argument and not follow directly from Contreras' result since $B_{q_0,k}$ is not a closed manifold and the Finsler metric associated with $\alpha_k$ degenerates on the boundary. We leave it to the interested reader to figure out if this proof strategy can be made to work.
	\end{rem}
	The magnetic system \eqref{e:magsys} has been studied in \cite{CFG} for $n=1$ and in \cite{MNsphere}, for arbitrary $n$. It fits into the larger class of magnetic systems induced by contact metric structures with symmetries, which have been studied on the three-torus \cite{MNtorus}, on Berger spheres \cite{IMBerger,IM24}, on the Heisenberg group \cite{EGMHeis,MNHeis} and the special linear group \cite{IMslr}, and on Sasakian manifolds \cite{Druta-Romaniuc2021}. The focus of these papers is, among other things, on finding periodic magnetic geodesics and on describing geometric properties of magnetic geodesics in terms of Hopf tubes and of higher curvatures. Magnetic systems with symmetries have also been studied in the context of integrable systems in \cite{Efi,BJ}, while their relation to symplectic invariants and the Mañe critical value has been investigated in \cite{Bim24b,Bim24,JB24,BM24} and \cite{CFP10}, respectively. The main novelty of Theorem \ref{t:mane} is, therefore, to compute the Mañé critical value for a class of magnetic systems with symmetries on high dimensional manifold (see also Theorem \ref{t:manegeneral}) and to characterize in a precise manner which pair of points are connected by a magnetic geodesic of a given subcritical energy. We refer to \cite{M24} for a generalization of \Cref{t:mane} to infinite-dimensional spheres, the completion of a regular Lie group, and an application to partial differential equations.
	
	\subsection{Magnetomorphisms and totally magnetic submanifolds}
	A key feature of the magnetic system $(\S,g,\d\alpha)$ is that it has a large group of symmetries and, as a byproduct, a large family of invariant submanifolds. The symmetries of a general magnetic system $(M,g,\sigma)$ are given by so-called \emph{magnetomorphisms} $F\colon M\to M$, that is, diffeomorphisms that preserve both $g$ and $\sigma$. As a consequence, magnetomorphisms send magnetic geodesics to magnetic geodesics with the same energy, see Proposition \ref{p:magneto}. On the other hand, invariant submanifolds are given by so-called \emph{totally magnetic submanifolds} $N\subset M$ which have the property that every magnetic geodesic that is tangent to $N$ is locally contained in $N$. Thus magnetomorphisms send totally magnetic submanifolds to totally magnetic submanifolds. 
	
	When $\sigma = 0$, totally magnetic submanifolds reduce to the classical notion of totally geodesic submanifolds. In Riemannian geometry, there are several equivalent characterizations of totally geodesic submanifolds; notably, a submanifold is totally geodesic if and only if its second fundamental form $\II$ vanishes. To develop an analogous infinitesimal characterization for totally magnetic submanifolds, we extend this classical condition by incorporating an additional constraint induced by the magnetic field $\sigma$. This leads to an infinitesimal criterion that characterizes totally magnetic submanifolds in the presence of a magnetic field.
	
	\begin{thm}\label{t:totmag}
Let $(M,g_M,\s_M)$ be a magnetic system. Let $N\subset M$ be a closed, embedded submanifold. Denote by $\iota\colon N\to M$ the inclusion map and by $g_N:=\iota^*g_M$ and $\s_N:=\iota^*\s_M$ the pullback metric and magnetic field. The following statements are equivalent:\begin{enumerate}
			\item The submanifold $N$ is totally magnetic in $(M,g_M,\s_M)$.
			\item If $\gc$ is a magnetic geodesic in $\left(N,g_N,\s_N\right)$ then $\iota\circ\gc$ is a magnetic geodesic in $\left(M,g_M,\s_M\right)$.
			%\item The magnetic second fundamental form of $\left(N,g_N,\s_N\right)$ vanishes identically: 
			%\[
			%\II^{\mathrm{mag}}_q(v)=0\quad \forall(q,v)\in TN.
			%\]
	\item The second fundamental form of $(N,g_N)$ vanishes identically 
		\[
		\II_q(v)=0\quad \forall (q,v)\in TN.
		\]
			and one of the following three equivalent conditions holds
			\begin{enumerate}
				\item the Lorenz force $Y^M$ of $(M,g_M,\s_M)$ along $N$ is equal to the Lorentz force $Y^N$ of $(N,g_N,\s_N)$: 
				\[
				Y^N_qv= Y^M_qv\quad \forall (q,v)\in TN;
				\]
				\item the Lorentz force $Y^M$ leaves the tangent bundle of $N$ invariant: 
				\[
				Y_q^Mv\in TN \quad\forall(q,v)\in TN;
				\] 
				\item the $g_M$-orthogonal of $TN$ is contained in the $\s_M$-orthogonal of $TN$:
				\[(\sigma_M)_q(u,v)=0 \quad \forall q\in N,\ u \in T_q N,\ v\in T_q N^{\perp}. 
				\]
			\end{enumerate}
		\end{enumerate}
	\end{thm}
	As an application, we classify magnetomorphisms and \emph{closed, connected} totally magnetic submanifolds with \emph{positive dimension} of our magnetic system of interest.
		\begin{cor}
		\label{c:symmetry}
		The group of magnetomorphisms of $\left(\S, g,\d\alpha\right)$ is the group of unitary matrices, that is, 
		\[
		\mathrm{Mag}(\S,g,\d\alpha)=U(n+1).
		\]
The lift of the magnetomorphism group to $T\S$ is Hamiltonian with respect to the twisted symplectic form $\omega_{\d\alpha}$. Its moment map is given by
\[
\mu\colon T\S\to \mathfrak u(n+1)^*,\qquad \mu(z,v)[A]:= g_z(Az,v)-\alpha_z(Az),\quad \forall\,A\in\mathfrak u(n+1),\ \forall\,(z,v)\in T\S.
\]
		A closed, connected submanifold $N$ of $\S$ of positive dimension is totally magnetic if and only if $N=\S\cap V$, where $V$ is a complex vector subspace of $\CC^{n+1}$. In this case, $N$ is magnetomorphic to $\SS^{2j+1}$ for some $j$, that is,
		\[
		\left(N,g_N,\sigma_N\right)\cong\left(\SS^{2j+1}, g,\d\alpha\right).
		\]
	\end{cor}
	\begin{rem}
		Corollary \ref{c:symmetry} illustrates the difference between totally magnetic submanifols and totally geodesic submanifolds. Indeed, every sphere $\SS^h$ obtained by intersecting $\S$ with a \emph{real} vector subspace of $\CC^{n+1}$ is totally geodesic, while only spheres $\SS^{2j+1}$ from intersections with \emph{complex} vector subspace of $\CC^{n+1}$ yield totally magnetic submanifolds.
	\end{rem}
As a special case of Corollary \ref{c:symmetry} we obtain a new proof of the following known fact \cite{Druta-Romaniuc2021}.  
	\begin{cor}
		\label{First main theorem}
		Each magnetic geodesic in $\left(\S, g,\d\alpha\right)$ is contained in a 3-sphere obtained intersecting the ambient sphere with a complex plane.\hfill\qed
	\end{cor}
		\subsection{Contact angle, Clifford tori and Hopf fibration.} 
		To continue our discussion, it is convenient to switch to the description of magnetic geodesics of $(\S,g,\d\alpha)$ with energy $k$ as magnetic geodesics of unit speed for the family of systems $(\S,g,s\d\alpha)$ where $s=\frac{1}{\sqrt{2k}}$ is now referred as the strength of the magnetic geodesic. In this description, the Mañé critical value corresponds to
		\begin{equation}\label{e:Mane value for unit speed magnetic geodesics on Sd}
		s_0:=\frac{1}{\sqrt{2\tfrac18}}=2.
		\end{equation}
		Among all symmetries given by Corollary \ref{c:symmetry}, the one corresponding to the Hopf flow \eqref{e:hopfflow} plays a special role since it corresponds to the center of the group $U(n+1)$. It is generated by $A=\mathrm i\in \mathfrak u(n+1)$ and yields the first integral
		\begin{equation}\label{e:psi}
		\mu(z,v)[\mathrm i]+\tfrac12=g_z(\mathrm{i}z,v)=:\cos\psi(z,v),
		\end{equation}
		so that $\psi=\psi(z,v)\in[0,\pi]$ is the angle between the tangent vector $v$ and the Reeb vector field $R_z=2\mathrm iz$ at $z$. We refer to $\psi$ as the \emph{contact angle}. For every $s\in\RR$, the magnetic geodesics of strength $s$ with $\psi=0$, respectively $\psi=\pi$, are exactly the orbits of the Reeb vector field with the positive, respectively, negative orientation. Magnetic geodesics with $\psi=\tfrac\pi2$ are tangent to the contact distribution $\ker\alpha$. By \cite[Proposition 3.1]{HR08}, it follows that as $s$ varies over $\RR$ magnetic geodesics tangent to $\ker\alpha$ are exactly the sub-Riemannian geodesics of $(\S,g,\ker\alpha)$. We can summarize this discussion in the following result.
		\begin{cor}
		The magnetic geodesic flow $(\S,g,\d\alpha)$ is an interpolation between the sub-Riemannian geodesic flow on $(\S,g,\ker\alpha)$ and the Reeb flow of $\alpha$ on $\S$, where the interpolation parameter is the contact angle $\psi$.\hfill\qed
		\end{cor}
We now move to analyze magnetic geodesics for an arbitrary value of the contact angle $\psi\in[0,\pi]$. By \cite{CFG}, magnetic geodesics are contained in two-dimensional Hopf tori (also called Clifford tori), and we want to describe how these tori relate to Mañé's critical value and the Hopf--Rinow theorem. We call a pair $w_0,w_1\in\C^{n+1}$ \emph{admissible} if 
\begin{equation}\label{e:w0w1}
\langle w_0,w_1\rangle=0,\qquad |w_0|^2+|w_1|^2=1.
\end{equation}
The \emph{width} of an admissible pair is the unique number $\tau\in[0,\pi]$ satisfying
\[
|w_0|=\cos(\tfrac\tau2),\quad|w_1|=\sin(\tfrac\tau2),
\]
or, equivalently, 
\begin{equation}\label{e:width}
|w_0|^2-|w_1|^2=\cos\tau.
\end{equation}
Notice that if $(w_0,w_1)$ is admissible with width $\tau$, then $(w_1,w_0)$ is admissible with width $\pi-\tau$. 

The two-dimensional \emph{Clifford torus} associated with the admissible pair $(w_0,w_1)$ is defined by
\[
\mathbb T^2_{w_0,w_1}:=\{ \lambda_0w_0+\lambda_1w_1\mid \lambda_0,\lambda_1\in \SS^1\subset\C\}.
\] 
We define the width of the torus as the width of the associated pair $(w_0,w_1)$. Observe that $\mathbb T^2_{w_0,w_1}=\mathbb T^2_{w_1,w_0}$, and therefore, the width of the torus is defined only up to the identification $\tau\sim \pi-\tau$. Moreover, if $|w_0|=1$, then $\mathbb T^2_{w_0,0}$ degenerates to the Hopf circle through $w_0$.

To put Clifford tori and magnetic geodesics in relation, we define for $s\geq 0$ the function
\begin{equation}
C_s(\psi):=(-\tfrac s2+\cos\psi,\sin\psi): [0,\pi] \to\IR^2
\end{equation}
and consider its image
\begin{equation}\label{e:Cs}
C_s([0,\pi])=\Big\{(-\tfrac s2+\cos\psi,\sin\psi)\mid \psi\in[0,\pi]\Big\}\subset \RR^2.
\end{equation}
The set $C_s([0,\pi])$ is the upper half-circle of radius $1$ in the plane with center $(-\tfrac s2,0)$, traversed counter-clockwise, see Figure \ref{pic:half_circles}. 
\begin{figure}[htb]
\centering
\includegraphics[width=0.5\textwidth]{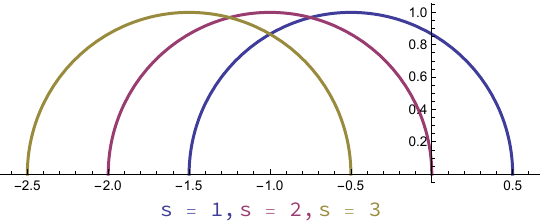}
\caption{The sets $C_s([0,\pi])$ for $s=1,2,3$.}
\label{pic:half_circles}
\end{figure}
Notice that for $0<s<2$, the origin in $\IR^2$ lies between the circle center and $C_s(0)$. The half-circle $C_2([0,\pi])$ meets the origin at $C_2(0)=(0,0)$.  For $s>2$, the origin lies to the right of $C_s(0)$. We denote by the functions
\[
|C_s|\colon [0,\pi]\to \big[1+\tfrac s2,|1-\tfrac s2|\big],\qquad \mathrm{Arg}(C_s)\colon[0,\pi]\to[0,\pi]
\]
the Euclidean norm of the point $C_s(\psi)$ and its angle with respect to the origin and the $x$-axis. For $s>0$, $|C_s|$ is strictly monotonically increasing. For $s<2$, $\mathrm{Arg}(C_s)$ is monotonically increasing and bijective. For $s=2$, $\mathrm{Arg}(C_s)$ is strictly monotonically increasing and with range $[\tfrac\pi2,\pi]$ upon setting $\mathrm{Arg}(C_2(0))=\tfrac\pi2$. For $s>2$, $\mathrm{Arg}(C_s)$ is decreasing on $[0,\psi_{\max}(s)]$ and increasing on $[\psi_{\max}(s),\pi]$ where 
\[
\psi_{\max}(s):=\arccos(\tfrac2s)
\]
and $\mathrm{Arg}(C_s(0))=\mathrm{Arg}(C_s(\pi))=\pi$, $\mathrm{Arg}(C_s(\psi_{\max}(s)))=\psi_{\max}(s)+\tfrac\pi2$.

Before presenting the next theorem, we recall for the reader's convenience that \( s = 2 \) corresponds exactly to Mañé's critical value of the magnetic system \( (\S, g, \d\alpha) \) by  \eqref{e:Mane value for unit speed magnetic geodesics on Sd}. Additionally, \( \psi \) represents the angle between the Reeb vector $R_z$ field and a unit tangent vector $(z,v)\in S\S$, as defined in \eqref{e:psi}.
\begin{thm}\label{t:clifford}
Let $S\S$ be the unit sphere bundle of $T\S$, and let $\widetilde{\mathcal M}_* := \psi^{-1}(0) \subset S\S$. There is a smooth map  
\[
(0,\infty) \times S\S \setminus \big(\{2\} \times \widetilde{\mathcal M}_*\big) \to \mathbb{C}^{n+1} \times \mathbb{C}^{n+1}, \qquad (s,z,v) \mapsto (w_0,w_1),
\]  
where, for a given $s$, $(w_0,w_1)$ can be determined as a function of $(z,v)$ by \eqref{e:w1w2}, and vice versa,  $(z,v)$ can be explicitly computed as a function of $(w_0,w_1)$ by \eqref{e:zv}. \\
Moreover, the magnetic geodesic $\gamma$ with strength $s$, initial condition $(z,v) \in S\S$, and contact angle $\psi$ is contained in the Clifford torus $\mathbb{T}^2_{w_0,w_1}$ with width 
\begin{equation}\label{e:tauc}
\tau=\mathrm{Arg}(C_s(\psi)).
\end{equation}
More precisely, $\gamma$ has the form
\begin{equation}\label{e:gammapar}
\gamma(t)=e^{\mathrm i\tfrac s2t}\Big(e^{-\mathrm i|C_s(\psi)|t}w_0+e^{\mathrm i|C_s(\psi)|t}w_1\Big),\quad\forall\,t\in\RR.
\end{equation}
In particular, for $\psi\in(0,\pi)$, then $\gamma$ has rotation number
\[
\rho_s(\psi):=\frac{\tfrac s2-|C_s(\psi)|}{\tfrac s2+|C_s(\psi)|}.
\]
The function $\rho_s\colon(0,\pi)\to (\rho^-_s,\rho^+_s)$ is bijective and monotonically decreasing, where
\begin{equation}\label{e:rho}
\rho^-_s:=-\tfrac{1}{1+s},\ \forall\,s\geq0,\qquad \rho^+_s:=\begin{cases}
	s-1,&\forall\,s\leq 2,\\
	\frac{1}{s-1},&\forall\,s\geq 2.\end{cases}
\end{equation}
Vice versa, if $\gamma$ satisfies \eqref{e:gammapar} for an admissible pair $(w_0,w_1)$, see \eqref{e:w0w1}, with width $\tau$ given by \eqref{e:tauc}, see also \eqref{e:width}, then $\gamma$ is a unit speed magnetic geodesic with strength $s$ and contact angle $\psi$.
\end{thm}
We use \eqref{e:gammapar} to give a condition for the existence of magnetic geodesics connecting two points, which we will use for the proof of Theorem \ref{t:mane}.
\begin{cor}\label{c:connecting}
Let $q_0,q_1\in\S$. There exists a magnetic geodesic $\gamma$ with strength $s$ and contact angle $\psi$ connecting $q_0$ and $q_1$, more precisely such that $\gamma(-\tfrac T2)=q_0$, $\gamma(\tfrac T2)=q_1$ for some $T\in\RR$, if and only if
\begin{equation}\label{e:time}
\langle q_0,q_1\rangle=e^{-\mathrm i\tfrac s2T}\Big(\cos\big(|C_s(\psi)|T\big)+\mathrm i\cos(\mathrm{Arg}(C_s(\psi)))\sin\big(|C_s(\psi)|T\big)\Big).
\end{equation}
\end{cor}
\begin{rem}
At \url{https://www.desmos.com/calculator/ldvfhhpkzl}, you can see the plot of the right-hand side of \eqref{e:time} as a function of $T\in[0,2\pi b|C_s(\psi)|^{-1}]$. The parameters of the plot are $b,s,a:=\cos\psi$.
\end{rem}
The number $|\langle q_0,q_1\rangle|$ and magnetic geodesics can be interpreted in terms of the Hopf map $\pi\colon \S\to\mathbb CP^n$. For simplicity, we restrict the discussion to a totally magnetic $\Sd$, that is, to $n=1$. In this case, there is an identification of $\mathbb C P^1$ with $\SS^2(\tfrac12)\subset\RR^3$, the Euclidean sphere of radius $\tfrac12$ which makes $\pi$ a Riemannian submersion. We denote by $g$ is the round metric on $\SS^2(\tfrac12)$ and $\sigma$ its area form. For $q_0,q_1\in\Sd$, the distance between $\pi(q_0)$ and $\pi(q_1)$ on $\SS^2(\tfrac12)$ is $\arccos\big(|\langle q_0,q_1\rangle|\big)$.
\begin{thm}\label{t:hopf}
 If $\gc$ is a unit speed magnetic geodesic on $(\Sd,g,s\d\alpha)$, $s\geq0$, with contact angle $\psi\in[0,\pi]$, then $\pi(\gc)$ is a constant speed magnetic geodesic on $\SS^2(\tfrac12)$. More precisely, after an orientation-preserving reparametrization $\pi(\gc)$ is a unit speed magnetic geodesic of $\big(\SS^2(\tfrac12),g,a_s(\psi)\sigma\big)$ with 
\begin{equation}\label{e:curvature}
%a_s\colon (0,\pi)\to \RR,\qquad a_s(\psi):=\frac{2}{\tan(\mathrm{Arg}(C_s(\psi)))}=\frac{2\cos\psi-s}{\sin\psi}.
a_s(\psi)=\frac{2}{\tan(\mathrm{Arg}(C_s(\psi)))}=\frac{2\cos\psi-s}{\sin\psi}\colon (0,\pi)\to \RR.
\end{equation}
In particular, the projected curve $\pi(\gc)$ is a geodesic circle of radius 	
\[
%r_s:[0,\pi]\longrightarrow \left[0,\tfrac{\pi}{2}\right],\quad 
r_s(\psi):=\tfrac12\mathrm{Arg}(C_s(\psi))\in \left[0,\tfrac{\pi}{2}\right].
\]
\end{thm}
\begin{rem}
The following formula for the radius of the geodesic circle $\pi(\gc)$ of arbitrary speed $c$ is of independent interest and has applications in \cite{BM24}. It is given by:
\[
R_{c,s,\delta} = \sqrt{\frac{(c-\delta)(c+\delta)}{s^2+4(s^2-s\delta)}}, \quad \delta := \cos \varphi,
\]
We leave the details of this computation to the reader. 
\end{rem}
We observe that the function $a_s$ is a bijection if and only if $0\leq s<2$, a fact that can be summarized in the following statement. 
\begin{cor}\label{S3 onto S2}
For fixed $s\in[0,2)$ the magnetic unit-speed magnetic geodesic flow on $(\Sd, g, s \d\alpha)$ covers \emph{all} magnetic systems $(\SS^2(\tfrac12),g,r\sigma)$, $r\in\RR$.%\hfill\qed
	\end{cor}
\begin{rem}
	The relationship between magnetic geodesics on $\Sd$ and magnetic geodesics on $\SS^2(\tfrac12)$ given in Theorem \ref{t:hopf} can be seen as an instance of the symplectic reduction for mechanical systems \cite[Theorem 4.3.3]{AM}, where the reduction is associated with the Hamiltonian $\SS^1$-action given by the Hopf flow.
\end{rem}
\begin{rem}
Taking $q_0=(1,0)\in\Sd$ and $q_1=(0,1)\in\Sd$, one can check that the ranges of $a_s$ and $r_s$ given in Theorem \ref{t:hopf} are compatible with the relationship between $\langle q_0,q_1\rangle$ and the existence of magnetic geodesics with strength $s$ connecting $q_0$ and $q_1$ described in Theorem \ref{t:mane}.
\end{rem}
	\begin{figure}
		\centering
		\includegraphics[width=0.7\textwidth]{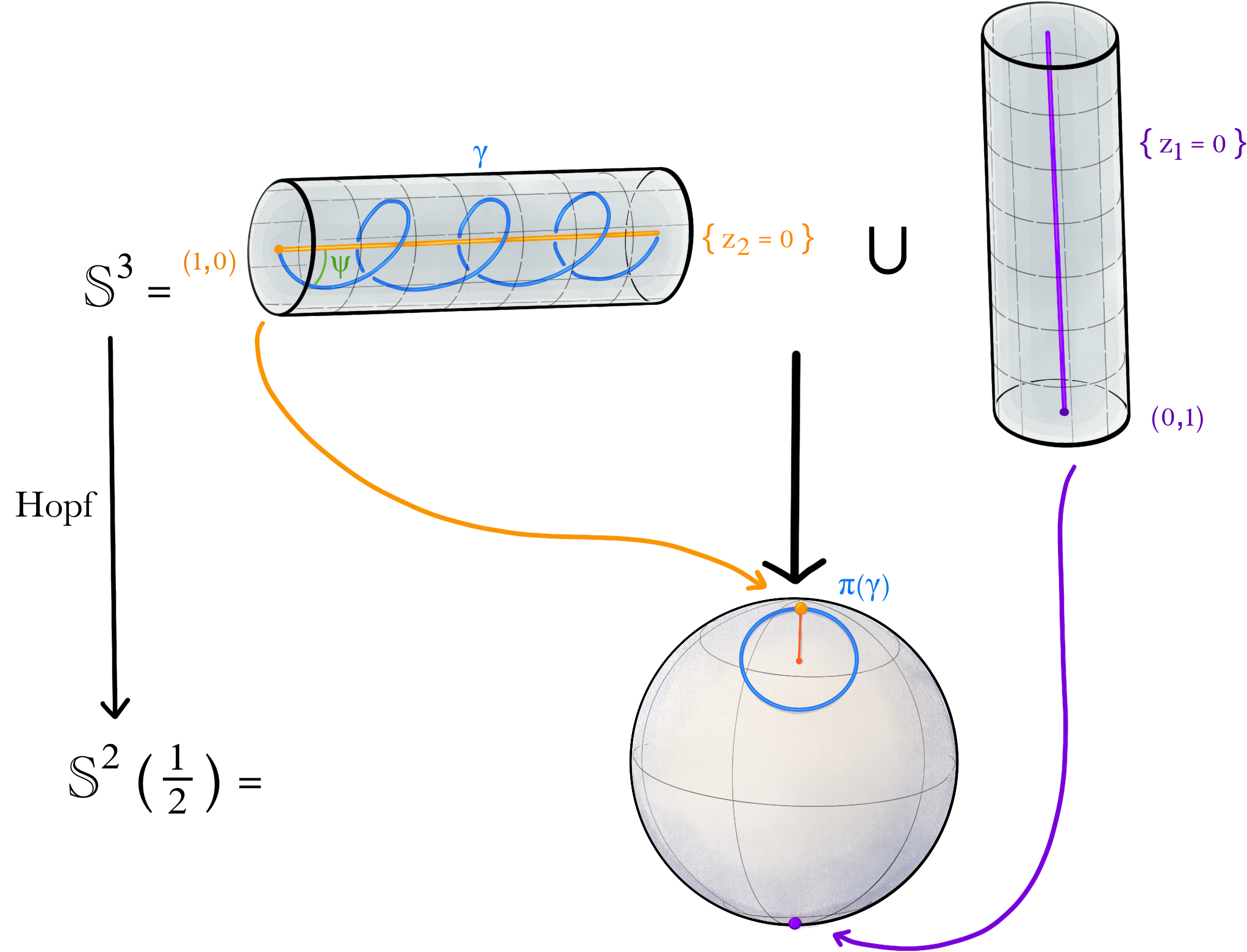}
		\caption{This picture illustrates Theorem \ref{t:hopf} and the fact that we can connect the Reeb orbit \textcolor{orange}{$\{z_2=0\}$} and the Reeb orbit \textcolor{purple}{$\{z_1=0\}$} with a unit speed magnetic geodesic of strength $s$ if and only if $s\in[0,2)$, see Theorem \ref{t:mane}. (Picture by Ana Chavez Caliz.)}
		\label{magnetic flow S3-> S2}
	\end{figure}

	\subsection{Applications to Hamiltonian PDE's and Hofer--Zehnder capacities:}
	To close this introduction let us mention two intriuiging and unexpected applications of the methods developed in this article. First, as already mentioned, in the recent work \cite{M24}, the third author uses the techniques we developed in \Cref{Section manes value} to prove an infinite-dimensional version of \Cref{t:mane} for the completion of a regular Lie group. In particular, he introduces the notion of Mañé's critical value for infinite-dimensional exact magnetic systems and illustrates this with a specific Hamiltonian partial differential equation — the so-called magnetic two-component Hunter--Saxton system. Notably, this system admits a formulation as an infinite-dimensional magnetic system, to which he generalizes \Cref{t:mane} in \cite[Thm. 7.1]{M24}. Second, the billiard system induced by the magnetic flow $(\Sd,g,s\d\alpha)$ on the three-sphere with the Hopf link removed is used in \cite{BM24} by Bimmermann and the third author to compute the Hofer--Zehnder capacity of the disk-tangent bundle of the lens spaces $L(p,1)$ for $p$ odd. In particular, this gives a counterexample of an analog of the Viterbo conjecture in the category of disk-tangent bundles. 
	\medskip

	\medskip
	
	\noindent\textbf{Acknowledgments:}
	We are grateful to Ana Chavez Caliz for providing Figure 1 and 2. We also would like to thank Christoph Schn\"orr for interesting discussions. We are indebted to the referee for their comments, which helped improve the paper.
	
	The authors acknowledge funding by the Deutsche Forschungsgemeinschaft (DFG, German Research Foundation) – 281869850 (RTG 2229), 390900948 (EXC-2181/1) and 281071066 (TRR 191). G.B. and L.M. gratefully acknowledge support from the Simons Center for Geometry and Physics, Stony Brook University at which some of the research for this article was performed during the program
	\textit{Mathematical Billiards: at the Crossroads of Dynamics, Geometry, Analysis, and Mathematical Physics.} 
	
\section{Mañé's critical value for magnetic Lagrangians}\label{Section manes value}
\subsection{Three equivalent definitions}
Let $(M,g)$ be a closed manifold $M$ with Riemannian metric $g$. Let $\alpha$ be a one-form on $M$ and consider the exact magnetic system $(M,g,\d\alpha)$. In this case, the magnetic geodesic flow $\Phi_{g,\sigma}$ is the Euler--Lagrange flow $\Phi_L$ of the magnetic Lagrangian
\[
L\colon TM\to\RR,\quad L(q,v):=\tfrac12|v|^2-\alpha_q(v),
\]
see \cite{Gin}. This means that $\gamma\colon[0,T]\to M$ is a magnetic geodesic if and only if $\gamma$ is a critical point of the action functional $S_L$ among all curves $\delta\colon [0,T]\to M$ such that $\delta(0)=\gamma(0)$ and $\delta(T)=\gamma(T)$. Here the action functional is defined as
\[
S_L(\gamma):=\int_0^TL(\gamma(t),\dot\gamma(t))\d t.
\]
This variational principle prescribes the length $T$ of the time interval and leaves free the energy of $\gamma$. On the other hand, given $k\in\RR$, $\gamma$ is a magnetic geodesic with energy $k$ if and only if $\gamma$ is a critical point of $S_{L+k}$ among all curves $\delta\colon [0,T']\to M$ with $\delta(0)=\gamma(0)$ and $\delta(T')=\gamma(T)$ for an arbitrary $T'>0$.\medskip

The Mañé's critical value of $L$ can be defined in three equivalent ways. First, $c(L)$ is the smallest energy value containing the graph of the differential of a function on $M$:
\begin{equation}\label{d:mane1}
		c(L)= \inf_{f\in C^\infty(M)}\sup_{q\in M}H(q,\d_qf),
	\end{equation}
where the magnetic Hamiltonian
\[
H\colon T^*M\to \RR,\quad H(q,p):=\tfrac12|p+\alpha_q|^2_q.
\]
is the Legendre dual of $L$. Notice that, when $M$ is simply connected, this definition coincide with the definition of the Mañe's critical value of the universal cover $c(M,g,\sigma)$ given in \eqref{e:maneuni}.

Second, $c(L)$ is the smallest energy value such that the free-period action functional with energy $k$ is bounded from below (by zero) on the space of loops:
\begin{equation}\label{d:mane2}
c(L)=\inf\{k\in\RR\mid S_{L+k}(\gamma)\geq 0,\ \forall\,T>0,\ \forall\,\gamma\colon [0,T]\to M,\ \gamma(0)=\gamma(T)\}.
\end{equation}
Third, $c(L)$ is minus the minimal action of invariant Borel probability measures $\mu$ for the flow $\Phi_L$ on $TM$:
\begin{equation}\label{d:mane3}
c(L)=-\inf_{\mu}S_L(\mu).
\end{equation}
Here 
\[
S_L(\mu):=\int_{TM}L(q,v)\d \mu
\]
denotes the action of the invariant probability measure $\mu$. For examples, if $\gamma\colon \RR/T\ZZ\to M$ is a periodic magnetic geodesic then the corresponding invariant measure $\mu_\gamma$ satisfies $S_L(\mu_\gamma)=\frac{S_L(\gamma)}{T}$. In general, if $\gamma$ is any magnetic geodesic, then by the Krylov–Bogolyubov theorem, there exists an invariant probability measure supported on the closure of the set $\{(\gamma(t),\dot\gamma(t))\mid t\in\RR\}$.

The Mather set of the system $\widetilde{\mathcal M}\subset TM$ is defined as the union of the supports of all action-minimizing probability measures. Dias Carneiro showed in \cite{DC} that $\widetilde{\mathcal M}\subset \{E=c(L)\}$, where $E$ is the kinetic energy of the system. Finally, Mather showed that $\widetilde{\mathcal M}$ is a graph over $M$ \cite{Mat}. Thus, we denote by $\mathcal M:=\pi_{TM}(\widetilde{\mathcal M})\subset M$, the projected Mather set.
\subsection{Mañé's critical value for Killing magnetic Lagrangians}
\begin{defn}\label{d:kill}
Let $g$ be a Riemannian metric and $\alpha$ be a one-form on a closed manifold $M$. We denote with $X$ the vector field on $M$ which is the metric dual of $\alpha$ with respect to $g$.
We say that $\alpha$ is Killing for $g$ if $X$ is a Killing vector field for $g$. If $\alpha$ is Killing for $g$, we call the associated Lagrangian $L$ and magnetic field $(M,g,\d\alpha)$ Killing as well. 
\end{defn} 
Recall that $X$ is a Killing vector field, when it generates a flow of isometries for $g$, or equivalently, $\mathcal L_Xg=0$, where $\mathcal L$ denotes the Lie derivative. By \cite[Proposition 2.2.1]{Pet}, this implies that
\[
\tfrac12g(\nabla_uX,v)=\d\alpha(u,v),\quad\forall\,u,v\in TM.
\]
By definition of Lorentz force \eqref{e:Lorentz}, it follows that
\[
Y=2\nabla X.
\]
Therefore, $Y\cdot 2X=\nabla_{2X}(2X)$, which means that the flow lines $\gamma\colon \RR\to M$ of $2X$ are magnetic geodesics by \eqref{e:mg}. As a result, the function $\tfrac12|X|^2=\tfrac12|\alpha|^2$ is constant along $\gamma$. We can now single out a special class of such flow lines.
\begin{lem}\label{l:kill}
Let $\alpha$ be Killing for $g$ and denote by $X$ the corresponding Killing vector field. Let $q$ be a point of $M$. The following properties are equivalent:
\begin{enumerate}
	\item $q$ is a critical point of the function $\tfrac12|\alpha|^2\colon M\to\RR$;
	\item $(\nabla_X X)_q=0$;
	\item $Y_q\cdot X_q=0$;
	\item $\d\alpha_q(X_q,\cdot)=0$;
	\item the Lorentz force $Y_q$ takes values in $\ker\alpha_q=X_q^\perp$;
	\item if $\gamma$ is a magnetic geodesic with $\gamma(0)=q$, then $\frac{\d}{\d t}\big|_{t=0}\alpha(\gamma)=\frac{\d}{\d t}\big|_{t=0}g_\gamma(X_\gamma,\dot\gamma)=0$.
\end{enumerate}
\end{lem}
\begin{proof}
Since $X$ is the dual of $\alpha$, we have $|\alpha|=|X|$. Moreover, 
\[
\d(\tfrac12|X|^2)=g(\nabla X,X)=2\d\alpha(\cdot,X)=-2\d\alpha(X,\cdot)=-g(\nabla_X X,\cdot)=-Y\cdot X.
\]
From this formula, we conclude that equivalence of the first five items of the statement. If $\gamma$ is a magnetic geodesic with $\gamma(0)=q$ and $\dot\gamma(0)=v$, we have
\[
\frac{\d}{\d t}\Big|_{t=0}\alpha(\dot\gamma)=g(\nabla_vX,v)+g(X,\nabla_{\dot\gamma}\dot\gamma)=2\d\alpha_q(v,v)+g(X,Y_qv)=g(X,Y_qv),
\]
from which the equivalence between the last two items of the statement follows.
\end{proof} 
Consider the function $\tfrac{1}{2}|\alpha|^2\colon M\to \RR$. Denote by $\tfrac12\Vert\alpha\Vert_\infty^2$ its maximum and by
\[
M_\alpha:=\{q\in M\mid \tfrac12|\alpha|_q^2=\tfrac12\Vert\alpha\Vert^2\}\subset M
\]
the subset of its maximizers
\begin{prop}\label{p:Malpha}
For all $r\in\RR$, every flow line of $rX$ starting at $q\in M_\alpha$ is entirely contained in $M_\alpha$ and it is both a standard and a magnetic geodesic.
\end{prop}
\begin{proof}
Let $q\in M_\alpha$ and let $\gamma$ be a flow line of $2X$ passing through $q$. Since $\tfrac12|\alpha|^2$ is constant along $\gamma$, we conclude that $\gamma$ is contained in $M_\alpha$. Hence, for all $t\in\RR$, $\gamma(t)$ is a critical point of $\tfrac12|\alpha|^2$ since $M_\alpha$ consists of critical points. It follows from the equivalence of the first three items in Lemma \ref{l:kill} that $\nabla_{\dot\gamma}\dot\gamma=0=Y_\gamma\cdot\dot\gamma$, and we conclude that $\gamma$ and all its constant speed reparametrizations are standard geodesics and magnetic geodesics at the same time.
\end{proof}
\begin{thm}\label{t:manegeneral}
Let $L$ be a Killing Lagrangian, then its Mañé critical value is
\[
c(L)=\tfrac12\Vert\alpha\Vert^2_\infty.
\]
If $X\colon M\to TM$ denotes the Killing vector field dual to $\alpha$, then 
\begin{itemize}
	\item the projected Mather set $\mathcal M$ is the union of the supports of all the invariant measures of the flow of $X|_{M_\alpha}$ (in particular, it contains all the periodic orbits of $X|_{M_\alpha}$);
	\item the Mather set is
	\[
	\widetilde{\mathcal M}=X(\mathcal M).
	\]
\end{itemize} 
\end{thm}
\begin{proof}
For all $(q,v)\in TM$, we compute
	\begin{equation}\label{e:belowL}
	L(q,v)=\tfrac12|v|^2-\alpha_q(v)\geq \tfrac12|v|^2-|\alpha_q||v|=\tfrac12\Big(|v|-|\alpha_q|\Big)^2+(\tfrac12\Vert\alpha\Vert_\infty^2-\tfrac12|\alpha_q|^2)-\tfrac12\Vert\alpha\Vert_\infty^2.
	\end{equation}
	Therefore, if $\mu$ is any Borel probability in $TM$, then
	\begin{equation}\label{e:equalityL}
	\int_{TM}L\d\mu\geq\int_{TM}\Big(-\tfrac12\Vert\alpha\Vert^2\Big)\d\mu=-\tfrac12\Vert\alpha\Vert^2_\infty. 
	\end{equation}
	From \eqref{d:mane3}, we deduce that $-\tfrac12\Vert\alpha\Vert^2_\infty\leq -c(L)$, that is, $c(L)\leq \tfrac12\Vert\alpha\Vert^2_\infty$.
	Moreover, a Borel probability measure $\mu$ achieves equality in \eqref{e:equalityL} if and only if $\mu$ is supported in the invariant set 
	\[
	\{(q,v)\in TM\mid v=X_q,\  |\alpha_q|=\Vert\alpha\Vert_\infty \}=X(M_\alpha).
	\]
Finally, a Borel probability measure $\mu$ is supported in $X(M|_{\alpha})$ and it is invariant under $\Phi_L$ if and only if its projection $\pi_{TM}(\mu)$ is supported in $M_\alpha$ and it is invariant under the flow of $X$. The statement follows.
\end{proof}
\begin{rem}
It is also possible to prove the theorem using the definition \eqref{d:mane2} of the critical value. The bound from above follows essentially from \eqref{e:belowL}. The bound from below follows by computing the action of a loop which is the concatenation of
\begin{itemize}
\item a flow line $t\mapsto \Phi_X^t(q)$ with $q\in M_\alpha$ for $t\in[0,T]$ with $T$ arbitrarily large;
\item a geodesics from $\Phi_X^T(q)$ to $q$ with uniformly bounded length.
\end{itemize}
In particular, the theorem holds also on certain non-compact manifolds. For instance, if $X|_{M_\alpha}$ has a periodic orbit.
\end{rem}
\subsection{Mañé's critical value for $K$-contact Lagrangians}
We can apply Theorem \ref{t:manegeneral} to the case of $K$-contact structures $(g,\alpha)$ on simply connected manifolds $M$ \cite{Blair}. In this case, $\alpha$ is a contact form and $g$ is a Riemannian metric with the following properties:
\begin{enumerate}
	\item the contact distribution $\ker\alpha$ and the Reeb vector field $R$ are orthogonal with respect to $g$;
	\item the Reeb vector field $R$ is a Killing vector field for $g$ with constant norm $r>0$.
\end{enumerate}
Recall that $\alpha$ is contact if $\alpha$ is nowhere zero and $\d\alpha$ is non-degenerate on the contact hyperplane distribution $\ker\alpha$. In this case, we can define the Reeb vector field $R$ by the equations $\d\alpha(R,\cdot)=0$ and $\alpha(R)=1$. Notice that if $(M,g,\alpha)$ is $K$-contact, then $\alpha$ is Killing for $g$ according to Definition \ref{d:kill}. Indeed, conditions (1) and (2) imply that the metric dual vector field of $\alpha$ is $X=\tfrac{1}{r^2} R$, which is Killing since $R$ is. We also observe that the function $\tfrac12|\alpha|^2=\tfrac12|X|^2=\tfrac{1}{2r^2}$ is constant and therefore, $M_\alpha=M$ and $c(L)=\tfrac{1}{2r^2}$. Finally, since $M_\alpha=M$, the equivalence between the first and the last item in Lemma \ref{l:kill} implies that $g_\gamma(X_\gamma,\dot\gamma)$ is constant along geodesics constant of motion. Since the speed of magnetic geodesics is constant as well, this is the same as saying that the angle $\psi\in[0,\pi]$ between a magnetic geodesic and the Reeb vector field is constant, where, compatibly with \eqref{e:psi},
\[
\cos\psi=\tfrac{1}{r|\dot\gamma|}g_\gamma(R_\gamma,\dot\gamma).
\]

We finish this section by observing that the magnetic system on the odd-dimensional sphere that we want to study is indeed $K$-contact.
\begin{thm}\label{t:firstpart}
Let $g$ be the standard Riemannian metric on $\S$ and let $\alpha$ be the standard contact form on $\S$, as described after \eqref{e:magsys}. Then $(g,\alpha)$ is $K$-contact on $\S$. The Mañé's critical value of the associated magnetic Lagrangian is $c(L)=\tfrac{1}{8}$ and the Mather set and projected Mather sets are
\[
\widetilde{\mathcal M}=\tfrac14R(\S)\subset T\S,\quad \mathcal M=\S.
\]
\end{thm}
\begin{proof}
For every $z\in \S$, we recall that $g_z=\mathrm{Re}\langle\cdot,\cdot\rangle$, $R_z=2\mathrm iz$, $\alpha_z=\tfrac12\mathrm{Re}\langle \mathrm iz,\cdot\rangle$, where $\langle\cdot,\cdot\rangle$ is the standard Hermitian form on $\C^{n+1}$. In particular, $|R_z|=2$ and $\alpha_z=\tfrac14g(R_z,\cdot)$ for all $z\in\S$, which shows property (1). Moreover, the Reeb flow is the Hopf flow \eqref{e:hopfflow} which takes values in $U(n+1)$. Therefore, the Hermitian product and hence $g$ are preserved under this flow, which shows that $R$ is Killing. This shows that the metric dual of $\alpha$ is $X=\tfrac14R$ and $c(L)=\tfrac12\tfrac{1}{2^2}=\tfrac18$. Finally, since all orbits of $R$ are periodic, we deduce that $\mathcal M=\S$ and hence $\widetilde{\mathcal M}=X(\S)=\tfrac14R(\S)$.
\end{proof}
With Theorem \ref{t:firstpart}, we have established the first part of Theorem \ref{t:mane}. In the next section, we are going to further specialize our discussion of the magnetic system $(\S,g,\d\alpha)$ and finish the proof of Theorem \ref{t:mane} as well as give the proof of Theorem \ref{t:clifford}.
 \section{Clifford tori and magnetic geodesics on $\S$}
This section will be entirely dedicated to the proof of Theorem \ref{t:clifford}.

From now on, we use the convention of considering magnetic geodesics $\gamma\colon\RR\to \S$ with unit speed for the family $(\S,g,s\d\alpha)$, for $s\geq 0$. Following \cite{MNsphere}, we rewrite the equation of magnetic geodesics \eqref{e:mg} in the ambient space $\C^{n+1}$.

For all $z\in\S$, we have the orthogonal projection 
\[
P_z\colon T_z\C^{n+1}\to T_z\S,\qquad P_zv=v-\mathrm{Re}\langle z,v\rangle z,\quad\forall\,v\in T_z\C^{n+1}.
\]
The Levi-Civita connection on $\S$ is obtained projecting the standard Levi-Civita connection on $\C^{n+1}$: 
\[
\nabla_{\dot\gamma}\dot\gamma=P_\gamma\ddot\gamma=\ddot\gamma-\mathrm{Re}\langle \gamma,\ddot\gamma\rangle \gamma.
\] 
Since $\gamma$ is contained in $\S$, we have that $\dot\gamma\in T_\gamma\S$ and therefore $0\equiv\mathrm{Re}\langle \gamma,\dot\gamma\rangle$. Differentiating this equation, we obtain
\[
0=\mathrm{Re}\langle \dot\gamma,\dot\gamma\rangle+\mathrm{Re}\langle \gamma,\ddot\gamma\rangle=1+\mathrm{Re}\langle \gamma,\ddot\gamma\rangle,
\]
where we used that $\gamma$ has unit speed. Hence,
\[
\nabla_{\dot\gamma}\dot\gamma=\ddot\gamma+\gamma.
\]
Similarly, for the Lorentz force, we have for all $z\in\S$ and $v_1,v_2\in T_z\S$:
\[
\d_z\alpha(v_1,v_2)=\mathrm{Re}\langle \mathrm iv_1 ,v_2\rangle=g_z(P_z\mathrm i v_1,v_2).
\]
This means that
\begin{equation}\label{e:YS}
Y_z=P_z\mathrm i=\mathrm i P_{(\C z)^\perp}=\mathrm i P_{\ker\alpha_z},
\end{equation}
where $P_{(\C z)^\perp}$ denotes the orthogonal projection onto the orthogonal of $\C z$, and we observe that this orthogonal coincides with the contact distribution at $z$.
Therefore,
\[
Y_zv_1=\mathrm iv_1+g_z(\mathrm iz,v_1)z,\quad \forall\,v_1\in T_z\S.
\]
If $\gamma$ is a magnetic geodesic with strength $s$, then by Lemma \ref{l:kill} the angle between $\dot\gamma$ and $R_z=2\mathrm iz$ is constant. Therefore, there exists $\psi\in[0,\pi]$ such that
\[
Y_\gamma\dot\gamma=\mathrm i\dot\gamma+(\cos\psi)\gamma.
\]
Putting these pieces together, we rewrite \eqref{e:mg} with strength $s$ as
\begin{equation}\label{e:linear}
\ddot\gamma -s\mathrm i\dot\gamma+(1-s\cos\psi)\gamma=0.
\end{equation}
Equation \eqref{e:linear} is linear and therefore its solutions can be determined solving the associated equation for the frequencies $\lambda=\mathrm i\theta\in\C$
\[
\theta^2-s\mathrm\lambda-(1-s\cos\psi)=0.
\]
The discriminant is
\[
\big(\tfrac s2\big)^2+(1-s\cos\psi)=(-\tfrac s2+\cos\psi)^2+\sin^2\psi,
\]
which is always non-negative and vanishes if and only if $s=2$ and $\psi=0$. 

Therefore, we have two real solutions $\theta_0(s,\psi)\leq \theta_1(s,\psi)$ which have the formula
\begin{equation}\label{e:thetas}
\theta_0(s,\psi)=\tfrac12s-|C_s(\psi)|,\quad \theta_1(s,\psi)=\tfrac12s+|C_s(\psi)|,
\end{equation}
where we have defined the point
\[
C_s(\psi):=(-\tfrac s2+\cos\psi,\sin\psi),
\]
which lies in the upper half-circle with center $(-\tfrac s2,0)$ and radius $1$, see \eqref{e:Cs}. The parameter $\psi$ represents the angle between the ray emanating from the center of the circle in direction of the positive $x$-axis and the ray passing through $C_s(\psi)$. 

Hence the solutions of \eqref{e:linear} have the form claimed in \eqref{e:gammapar}:
\[
\gamma(t)=e^{\mathrm i\tfrac s2t}\Big(e^{-\mathrm i|C_s(\psi)|t}w_0+e^{\mathrm i|C_s(\psi)|t}w_1\Big),\quad\forall\,t\in\RR.
\] 
for some $w_0,w_1\in\C^{n+1}$. Let us suppose that $\gamma(0)=z$ and $\dot\gamma(0)=v$ and let us look for a formula for $w_0$ and $w_1$. We compute
\[
\dot\gamma(t)=\mathrm i\tfrac s2\gamma(t)+\mathrm i|C_s(\psi)|e^{\mathrm i\tfrac s2t}\Big(-e^{-\mathrm i|C_s(\psi)|t}w_0+e^{\mathrm i|C_s(\psi)|t}w_1\Big),\quad\forall\,t\in\RR.
\]
Setting $t=0$, we get,
\begin{equation}\label{e:zv}
z=w_0+w_1,\quad v=\mathrm i\tfrac s2 (w_0+w_1)+\mathrm i|C_s(\psi)|(w_1-w_0).
\end{equation}
Solving this system for $w_0$ and $w_1$, we get
\begin{equation}\label{e:w1w2}
w_0=\frac{1}{2\mathrm{i}|C_s(\psi)|}\Big(\mathrm{i}(|C_s(\psi)|+\tfrac s2)z-v\Big),\qquad w_1=\frac{1}{2\mathrm{i}|C_s(\psi)|}\Big(\mathrm{i}(|C_s(\psi)|-\tfrac s2)z+v\Big).
\end{equation}
Therefore, we see that $w_0$ and $w_1$ depend smoothly from $z$ and $v$ as soon as $|C_s(\psi)|\neq0$. This last condition is equivalent to $s=2$ and $\psi=0$. as mentioned above.
We compute
\[
\langle w_0,w_1\rangle=\frac{1}{4|C_s(\psi)|^2}\Big(|C_s(\psi)|^2-\tfrac{s^2}{4}-1+(|C_s(\psi)|+\tfrac s2)\langle \mathrm{i}z,v\rangle-(|C_s(\psi)|-\tfrac s2)\overline{\langle \mathrm iz,v\rangle}\Big) 
\]
We have $\mathrm{Re}\langle \mathrm iz,v\rangle=\cos\psi$ and $\mathrm{Im}\langle \mathrm iz,v\rangle=\mathrm{Re}\langle z,v\rangle=0$. The expression within parentheses becomes
\[
|C_s(\psi)|^2-\tfrac{s^2}{4}-1+s\cos\psi
\]
which is equal to zero by definition of $C_s(\psi)$. Therefore, we find that
\[
\langle w_0,w_1\rangle=0
\]
and taking the norm of the first equation in \eqref{e:zv}, we also get
\[
|w_0|^2+|w_1|^2=1.
\]
Recall the definition of the width $\tau\in[0,\pi]$ as $|w_0|=\cos(\tfrac\tau2)$. We obtain from \eqref{e:zv},
\[
\cos\psi=\mathrm{Re} \langle \mathrm iz,v\rangle=\tfrac s2+|C_s(\psi)|(|w_1|^2-|w_0|^2)=\tfrac s2+|C_s(\psi)|\cos\tau.
\]
Hence, $\cos\tau=\frac{-\tfrac s2+\cos\psi}{|C_s(\psi)|}$. Since the cosine function is bijective on $[0,\pi]$, we deduce that
\[
\tau=\mathrm{Arg}(C_s(\psi))
\]
is the angle made by the vector $C_s(\psi)$ with the positive $x$-axis. From the expression of $\gamma$, see \eqref{e:gammapar}, we also get that the frequency of rotation in the $w_0$-factor is $\tfrac{s}{2}-|C_s(\psi)|$, while the frequency of rotation in the $w_1$-factor is $\tfrac{s}{2}+|C_s(\psi)|$. This shows that the rotation number of the orbit is 
\[
\rho_s(\psi)=\frac{\tfrac{s}{2}-|C_s(\psi)|}{\tfrac{s}{2}+|C_s(\psi)|}.
\]
It is immediate to see that $\rho_s\colon(0,\pi)\to(\rho_s^-,\rho_s^+)$ is bijective and motonically decreasing, where $\rho_s^-$ and $\rho_s^+$ are defined in \eqref{e:rho}.

To finish the proof of Theorem \ref{t:clifford}, we need to prove that if $\gamma$ is given by \eqref{e:gammapar} with $w_0,w_1\in\C^{n+1}$ satisfying \eqref{e:w0w1} and with the corresponding Clifford torus having width satisfying \eqref{e:tauc}. There are two possible approaches. The first, more elementary, approach consists in checking that $\gamma$ given by \eqref{e:gammapar} is a unit speed curve contained in $\S$ satisfying \eqref{e:linear}. We leave this computation to the reader. In the second approach, we use the fact that the group $U(n+1)$ sends magnetic geodesics with strength $s$ and contact angle $\psi$ to magnetic geodesics with strength $s$ and contact angle $\psi$ by Corollary \ref{c:symmetry} and Proposition \ref{p:magneto}. The result follows by observing that there exists at least one magnetic geodesic with strength $s$ and contact angle $\psi$ and the fact that $U(n+1)$ acts transitively on the set of pairs $(w_0,w_1)$ satisfying \eqref{e:w0w1} and having prescribed width.
\begin{rem}
In the argument above, we used that the speed $|v|$ (in this case normalized to $1$) and the angle $\psi$ given by $\cos\psi=\mathrm{Re}\langle \mathrm iz,v\rangle$
are preserved along magnetic geodesics. As a sanity check, the reader can prove that if $\gamma$ is a solution to the equation \eqref{e:linear} such that $\gamma(0)=z\in \S$ and $\dot\gamma(0)=v\in T_z\S$, then $|\gamma|=|\dot\gamma|\equiv 1$ and $\mathrm{Re}\langle \mathrm i\gamma,\dot\gamma\rangle\equiv\cos\psi$. In particular, $\gamma$ is a magnetic geodesic with strength $s$ and contact angle $\psi$.
\end{rem}
\section{The Hopf--Rinow Theorem for $(\S,g,\d\alpha)$}\label{s:hr}

In this section, we prove Corollary \ref{c:connecting} and use it to establish the second part of Theorem \ref{t:mane}.
\subsection{Proof of Corollary \ref{c:connecting}}
Let $q_0,q_1$ be two points in $\S$. From Theorem \ref{t:clifford}, we know that there exists a magnetic geodesic $\gamma\colon\RR\to\S$ with strength $s\geq 0$ such that $\gamma(-\tfrac T2)=q_0$ and $\gamma(\tfrac T2)=q_1$ for some $T\geq 0$ if and only if there exists $\psi\in[0,\pi]$ and an admissible pair $w_0,w_1\in\C^{n+1}$ with width $\tau=\mathrm{Arg}(C_s(\psi))$ satisfying
\begin{equation}\label{e:q0q1}
	\begin{aligned}
		q_0&=e^{-\mathrm i\tfrac{sT}{4}}\Big(e^{\mathrm i|C_s(\psi)|\tfrac T2}w_0+e^{-\mathrm i|C_s(\psi)|\tfrac T2}w_1\Big),\\
		q_1&=e^{\mathrm i\tfrac{sT}{4}}\Big(e^{-\mathrm i|C_s(\psi)|\tfrac T2}w_0+e^{\mathrm i|C_s(\psi)|\tfrac T2}w_1\Big).
	\end{aligned}
\end{equation}
Assume now that \eqref{e:q0q1} holds. Taking the Hermitian product of the two right-hand sides, we get
\begin{align*}
	\langle q_0,q_1\rangle&=e^{-\mathrm i\tfrac{sT}{2}}\Big(e^{\mathrm i|C_s(\psi)|T}|w_0|^2+e^{-\mathrm i|C_s(\psi)|T}|w_1|^2\Big)\\
	&=e^{-\mathrm i\tfrac{sT}{2}}\Big(\cos(|C_s(\psi)|T)(|w_0|^2+|w_1|^2)+\mathrm i\sin(|C_s(\psi)|T)(|w_0|^2-|w_1|^2)\Big)\\
	&=e^{-\mathrm i\tfrac{sT}{2}}\Big(\cos(|C_s(\psi)|T)+\mathrm i\cos(\mathrm{Arg}(C_s(\psi)))\sin(|C_s(\psi)|T)\Big),
\end{align*}
where in the last step, we used the definition \eqref{e:width} of the width $\tau$ and equation \eqref{e:tauc}. Thus, we have established equation \eqref{e:time} in Corollary \ref{c:connecting}.

Conversely, assume that equation \eqref{e:time} holds and let us show that we can find an admissible pair $w_0,w_1$ with width $\tau=\mathrm{Arg}(C_s(\psi))$ satisfying \eqref{e:q0q1}. We distinguish two cases. If $\sin(|C_s(\psi)|T)\neq0$, then we can solve the linear system \eqref{e:q0q1} for $w_0$ and $w_1$ and define
\begin{equation}
	\begin{aligned}
		w_0&:=\frac{1}{2\mathrm i\sin(|C_s(\psi)|T)}\Big(e^{\mathrm i(\tfrac{s}{2}+|C_s(\psi)|)\tfrac T2}q_0-e^{-\mathrm i(\tfrac{s}{2}+|C_s(\psi)|)\tfrac T2}q_1\Big),\\
		w_1&:=\frac{1}{2\mathrm i\sin(|C_s(\psi)|T)}\Big(-e^{\mathrm i(\tfrac{s}{2}-|C_s(\psi)|)\tfrac T2}q_0+e^{-\mathrm i(\tfrac{s}{2}-|C_s(\psi)|)\tfrac T2}q_1\Big).
	\end{aligned}
\end{equation}
It is a simple computation using the fact that $q_0,q_1\in\S$ and that \eqref{e:time} holds to show that the pair $w_0,w_1$ is admissible and has width $\tau=\mathrm{Arg}(C_s(\psi))$.

We consider the case $\sin(|C_s(\psi)|T)=0$. Hence, either $\cos(|C_s(\psi)|T)=1$ or $\cos(|C_s(\psi)|T)=-1$. We treat the first subcase only since the second is analogous. Equation \eqref{e:time} becomes
\[
\langle q_0,q_1\rangle=e^{-\mathrm i\tfrac{sT}{2}}.
\]
By Cauchy--Schwarz, this implies that
\begin{equation}\label{e:cs}
q_1=e^{\mathrm i\tfrac{sT}{2}}q_0.
\end{equation}
The system \eqref{e:q0q1} becomes
\[
u_0=w_0+w_1,\qquad u_0:=\lambda q_0=\bar\lambda q_1,\qquad \lambda:=e^{\mathrm i(\tfrac{s}{2}-|C_s(\psi)|)\tfrac T2}.
\]
The equation $\lambda q_0=\bar\lambda q_1$ is satisfied because of \eqref{e:cs} and the assumption $e^{\mathrm i|C_s(\psi)|T}=1$. All admissible pairs $w_0,w_1$ with width $\tau=\mathrm{Arg}(C_s(\psi))$ solving the equation $u_0=w_0+w_1$ with $u_0\in\S$ can be constructed as follows. Let $u_1\in\S$ be any vector such that $\langle u_0,u_1\rangle=\cos\tau$. Consider the solutions $w_0,w_1$ to the system
\[
u_0=w_0+w_1,\quad u_1=w_0-w_1,
\] 
that is, $w_0=\tfrac12(u_0+u_1)$ and $w_1=\tfrac12(u_0-u_1)$. It is a simple computation using that $u_0,u_1\in\S$ and $\langle u_0,u_1\rangle=\cos\tau$ to show that the pair $w_0,w_1$ is admissible and has width $\tau$. This finishes the proof of Corollary \ref{c:connecting}

\subsection{Proof of Theorem \ref{t:mane} (Second part)}\label{ss:mane2}
Let $q_0,q_1\in\S$. Using \eqref{e:time}, we want to determine for which energies $k$, there is a magnetic geodesic with unit speed and strength $s$ connecting $q_0$ and $q_1$. This yields a corresponding result for magnetic geodesics with kinetic energy $k=\tfrac{1}{2s^2}$. We let $\lambda:=\langle q_0,q_1\rangle$. We distinguish three cases for $\lambda$. First, if $|\lambda|=1$, then $q_0$ and $q_1$ are on the same Hopf trajectory. Since every Hopf trajectory is a standard geodesic and it is also a magnetic geodesic with any strength by Proposition \ref{p:Malpha}, we conclude that $q_0$ and $q_1$ are connected by a magnetic geodesic for every strength $s$. Second, if $\lambda=0$, then, by \eqref{e:time}, there is a magnetic geodesic with strength $s$ and contact angle $\psi\in(0,\pi)$ connecting $q_0$ and $q_1$ in time $T$ if and only if
\[
\cos(|C_s(\psi)|T)=0,\quad \mathrm{Arg}(C_s(\psi))=\tfrac\pi2.
\]
The second equation has a solution $\psi$ if and only if $s<2$. If $s<2$, then $T=\tfrac{1}{|C_s(\psi)|}(\tfrac\pi2+h\pi)$ for $h\in\ZZ$ are solutions to the first equation.

Third, we assume that $0<|\lambda|<1$. In this case, we are going to use that, by \eqref{e:time}, we have
\begin{equation}\label{e:|lambda|}
|\lambda|^2=\cos^2(|C_s(\psi)|T)+\cos^2(\mathrm{Arg}(C_s(\psi)))\sin^2(|C_s(\psi)|T),
\end{equation}
In particular,
\begin{equation}\label{e:|lambda|2}
	|\lambda|\geq \big|\cos(\mathrm{Arg}(C_s(\psi))\big|.
\end{equation}
Recall that $\psi\mapsto \mathrm{Arg}(C_s)\colon[0,\pi]\to[0,\pi]$ is
\begin{enumerate}
	\item bijective, monotonically increasing and continuous, for $0<s<2$;
	\item continuous, strictly monotonically increasing on $[0,\psi^{\max}_s]$, strictly monotonically decreasing on $[\psi^{\max}_s,\pi]$ for $s\geq 2$, where $\psi^{\max}_s=\arccos(\tfrac2s)$.
\end{enumerate}
By \eqref{e:Cs}, we have 
\[
\cos(\mathrm{Arg}(C_s(\psi^{\max}_s)))=\frac{-\tfrac s2+\tfrac2s}{\sqrt{\tfrac{s^2}{4}-1}}=-\sqrt{1-\tfrac{4}{s^2}}.
\]
Combining this formula with \eqref{e:|lambda|2}, we get
\begin{equation}\label{e:|lambda|3}
|\lambda|\geq \sqrt{1-\tfrac{4}{s^2}}.
\end{equation}
We distinguish two cases.
\begin{enumerate}
	\item If $|\lambda|=\sqrt{1-\tfrac{4}{s^2}}$, then \eqref{e:|lambda|} is satisfied if and only if $\psi=\psi^{\max}_s$ and $|C_s(\psi^{\max}_s)|T=\tfrac\pi2+m\pi$ for some $m\in\ZZ$.
	Plugging in these values in \eqref{e:time}, we get
	\[
	\lambda=\sqrt{1-\tfrac{4}{s^2}}e^{\mathrm i( a_s+mb_s)},\ \ m\in\ZZ,\qquad a_s:=-\tfrac\pi2\big(1+\tfrac{s}{2|C_s(\psi^{\max}_s)|}\big), \qquad b_s:=\pi\big(1-\tfrac{s}{2|C_s(\psi^{\max}_s)|}\big).
	\]
	\item If $|\lambda|>\sqrt{1-\tfrac{4}{s^2}}$, then there exists a closed interval $I_s\subset[0,\pi]$ of positive length such that \eqref{e:|lambda|3} holds if and only if $\psi\in I_s$. For $\psi\in I_s$, we rewrite \eqref{e:|lambda|} as \[
	\cos^2(|C_s(\psi)|T)=\frac{|\lambda|^2-\cos^2(\mathrm{Arg}(C_s(\psi)))}{1-\cos^2(\mathrm{Arg}(C_s(\psi)))}.
	\]
	and define
	\[
	T(\psi):=\frac{1}{|C_s(\psi)|}\arccos\left(\sqrt{\frac{|\lambda|^2-\cos^2(\mathrm{Arg}(C_s(\psi)))}{1-\cos^2(\mathrm{Arg}(C_s(\psi)))}}\right).
	\]
	\end{enumerate}
Denoting with $f(T)$ the right-hand side of \eqref{e:time}, there is a continuous function $\eta_s\colon I_s\to\RR$ such that 
\[
f(T(\psi))=|\lambda|e^{-\mathrm i\eta_s(\psi)}.
\]
Hence, for all $m\in\ZZ$, we get from \eqref{e:time} that
\[
f\Big(T(\psi)+\tfrac{2\pi}{|C_s(\psi)|} m\Big)=e^{-\mathrm i\tfrac{2\pi s}{|C_s(\psi)|}m}f(T(\psi))=e^{-\mathrm i\eta_{s,m}(\psi)}|\lambda|,\quad \eta_{s,m}(\psi):=\tfrac{2\pi s}{|C_s(\psi)|}m+\eta_s(\psi).
\] 
Our aim is to show that there exists $\psi\in I_s$ and $m\in\NN$ such that $e^{-\mathrm i\eta_{s,m}(\psi)}|\lambda|=\lambda$. To this purpose, we show that for $|m|$ large enough the target of the function $\eta_{s,m}\colon I_s\to\RR$ contains an interval of length at least $2\pi$. This is a consequence of the fact that the function $|C_s|$ is strictly increasing for $s>0$, which we observed in the introduction. Take $\psi_0,\psi_1\in I_s$ such that $|C_s(\psi_0)|\neq |C_s(\psi_1)|$. We have
\[
|\eta_{s,m}(\psi_1)-\eta_{s,m}(\psi_0)|\geq 2\pi s \left|\tfrac{1}{|C_s(\psi_1)|}-\tfrac{1}{|C_s(\psi_0)|}\right||m|-\Big(\max_{\psi\in I_s}\eta_s(\psi)-\min_{\psi\in I_s}\eta_s(\psi)\Big)
\]
and, for every $s>0$, the right-hand side is at least $2\pi$ if $|m|$ is large enough. By the Intermediate Value Theorem, the target of $\eta_{s,m}$ contains an interval of length at least $2\pi$. This finishes the proof of the second part of Theorem \ref{t:mane}.

\section{Magnetic geodesics and the Hopf fibration}\label{s:hopf}
In this section we consider the magnetic system $(\Sd,g,s\d\alpha)$. By Corollary \ref{c:symmetry}, we know that every magnetic geodesic of $(\S,g,s\d\alpha)$ is contained in a totally magnetic three-sphere whose associated magnetic system is magnetomorphic to $(\Sd,g,s\d\alpha)$. Therefore, Theorem \ref{t:hopf} also sheds light to the standard magnetic system on spheres of every odd dimension.

For $a\in\RR$, we define the standard magnetic system $(\SS^2(\tfrac12),g,a\sigma)$, where $\SS^2(\tfrac12)$ is the Euclidean two-sphere of radius $\tfrac12$ in $\RR^3$, $g$ is the associated Euclidean metric, and $\sigma$ is the area form of $g$. By \cite{BKmag}, unit magnetic geodesics of this system are geodesic circles on $\SS^2(\tfrac12)$ of radius $r\in[0,\tfrac\pi2]$ such that 
\begin{equation}\label{e:r}
\tan (2r)=\frac{2}{a}.
\end{equation}
The group of magnetomorphisms of $(\SS^2(\tfrac12),g,a\sigma)$ is $SO(3)$, the group of orientation preserving isometries of $g$.

The Hopf map is given by 
\begin{equation}\label{e:hopf}
	\pi\colon \Sd \longrightarrow \SS^2(\tfrac12)\subset \C\times\RR\equiv \RR^3,\qquad \pi(z):=\tfrac{1}{2}\left(2\bar z_1z_2, |z_1|^2-|z_2|^2\right).
\end{equation}
The Hopf map is a Riemannian submersion between the standard metrics on $\Sd$ and $\SS^2(\tfrac12)$. Moreover, $\pi^*\sigma=\d\alpha$. Every element of $U(2)$ yields a magnetomorphism $A\colon\Sd\to\Sd$ which descends via $\pi$ to a magnetomorphism $\bar A\colon \SS^2(\tfrac12)\to\SS^2(\tfrac12)$. The map
\[
U(2)\to SO(3),\quad A\mapsto\bar A
\]
is surjective and its kernel is given by the multiples of the identity, that is, by the Hopf flow.

\subsection{The proof of Theorem \ref{t:hopf}} Let $q_0,q_1\in\Sd$. We want to show that
\begin{equation}\label{e:dist}
\mathrm{dist}(\pi(q_0),\pi(q_1))=\arccos(|\langle q_0,q_1\rangle|),
\end{equation}
where $\mathrm{dist}(\pi(q_0),\pi(q_1))$ denotes the distance between $\pi(q_0)$ and $\pi(q_1)$ on $\SS^2(\tfrac12)$. Upon acting with a magnetomorphism of $\Sd$, we assume that $q_0=(1,0)\in \Sd$ while $q_1=(z_1,z_2)$ is arbitrary. We have
\[
|\langle q_0,q_1\rangle|=|z_1|.
\]
On the other hand, using \eqref{e:hopf}, we see that
$\pi(q_0)=(0,0,1)$ is the north pole and, therefore, 
\[
\mathrm{dist}(\pi(q_0),\pi(q_1))=\tfrac12\arccos(|z_1|^2-|z_2|^2)=\tfrac122\arccos(|z_1|)=\arccos(|\langle q_0,q_1\rangle|),
\]
where we used the identity $|z_1|^2+|z_2|^2=1$ and the duplication formula for the cosine. This establishes \eqref{e:dist}.

We move on to investigate the projection $\pi(\gamma)\colon\to\SS^2(\tfrac12)$ of magnetic geodesics of strength $s$ and contact angle $\psi$. For this purpose, we are going to use formula \eqref{e:gammapar}:
\[
\gamma(t)=e^{\mathrm i\tfrac s2t}\Big(e^{-\mathrm i|C_s(\psi)|t}w_0+e^{\mathrm i|C_s(\psi)|t}w_1\Big),\quad\forall\,t\in\RR,
\]
where $w_0,w_1\in\C^2$ is an admissible pair with width $\tau=\mathrm{Arg}(C_s(\psi))$. Upon acting with an element of $U(2)$, we can assume that $w_0=\cos(\tfrac\tau2)(1,0)$ and $w_1=\sin(\tfrac\tau2)(0,1)$. In this case,
\[
\pi(\gamma)=\tfrac12\Big(2\cos(\tfrac\tau2)\sin(\tfrac\tau2)e^{\mathrm i2|C_s|t},\cos^2(\tfrac\tau2)-\sin^2(\tfrac\tau2)\Big)=\tfrac12\Big(\sin\tau e^{\mathrm i2|C_s|t},\cos\tau\Big),
\]
which is a constant speed parametrization with positive orientation of a geodesic circle on $\SS^2(\tfrac12)$ of radius
\[
r=\tfrac12\tau=\tfrac12\mathrm{Arg}(C_s(\psi)).
\] 
Using \eqref{e:r}, we see that $\pi(\gamma)$ is a magnetic geodesic for $(\SS^2(\tfrac12),g,a_s(\psi)\sigma)$ with
\[
a_s(\psi):=\frac{2}{\tan(\mathrm{Arg}(C_s(\psi)))},
\]
as we wanted to show. This finishes the proof of Theorem \ref{t:hopf}.
\begin{rem}
Unit speed magnetic geodesics of $(\SS^2(\tfrac12),g,a\sigma)$ are curves with geodesic curvature $\kappa$ equal to $a$ \cite{BKmag}. Therefore, by \eqref{e:curvature}, we see that if $\gamma$ is a unit speed magnetic geodesic in $\Sd$ with strength $s$ and contact angle $\psi$, then $\pi(\gamma)$ has geodesic curvature
\[
\kappa=\frac{2\cos\psi-s}{\sin\psi}.
\]
For instance, if $\gamma$ is tangent to the contact distribution, that is, $\psi=\tfrac\pi2$, we get $\kappa=-s$, which, up to a normalizatin factor of $-2$, coincides with the curvature of a horizontal sub-Riemannian geodesic for $(\Sd,g,\ker\alpha)$ as defined in \cite[§3]{HR08}.
\end{rem}
\begin{rem}
We know that, up to the action of $U(2)$, every unit speed magnetic geodesic of strength $s$ and contact angle $\psi$ can be written as
\[
\gamma(t)=e^{\mathrm i\theta_0t}\cos(\tfrac{\tau}{2})\zeta_0+e^{\mathrm i\theta_1t}\sin(\tfrac{\tau}{2})\zeta_1=\Big(e^{\mathrm i\theta_0t}\cos(\tfrac{\tau}{2}),e^{\mathrm i\theta_1t}\sin(\tfrac{\tau}{2})\Big),\qquad \forall\,t\in\RR,
\]
where $\zeta_0=(1,0)$, $\zeta_1=(0,1)$, $\tau=\mathrm{Arg}(C_s(\psi))$, $\theta_0=\tfrac s2-|C_s(\psi)|$, $\theta_1=\tfrac s2+|C_s(\psi)|$.

The curve $\gamma$ is cointained in the Clifford torus
\[
\mathbb T^2_{\tau}:=\mathbb T^2_{\cos(\tfrac\tau2)\zeta_0,\sin(\tfrac\tau2)\zeta_1},
\]
which is the common boundary of the two solid tori
\[
V_{1,\tau}:=\left\{(z_1,z_2)\in \Sd \mid \vert z_2 \vert \leq \cos(\tfrac\tau2) \right\},\qquad V_{2,\tau}:=\left\{(z_1,z_2)\in \Sd \mid \vert z_1 \vert \leq \sin(\tfrac\tau2) \right\}
\]
In other words,
\[
\Sd=V_{1,\tau}\cup V_{2,\tau},\qquad \mathbb T^2_\tau=V_{1,\tau}\cap V_{2,\tau}.
\]
The two solid tori have souls \[
S_1:= \SS^1\times\{0\}\subseteq \Sd \quad \text{and}\quad S_2:= \{0\}\times\SS^1 \subseteq \Sd,
\]
which are the Hopf fibres, equivalently Reeb orbits, through the points $\zeta_0$ and $\zeta_1$. This discussion can be visualized in the following picture, see also \cite{CFG}.
\begin{figure}
	\centering
	\includegraphics[width=\textwidth]{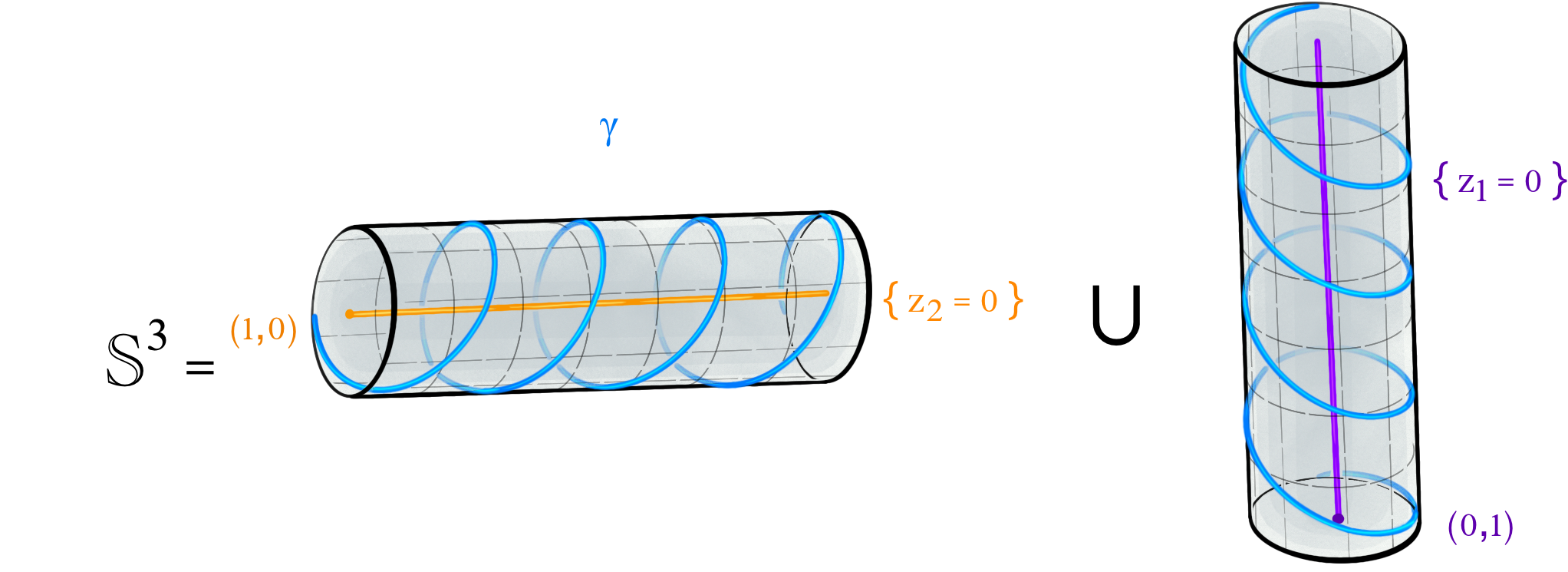}
	\caption{$\Sd$ is the union of two solid tori. The magnetic geodesic \textcolor{blue}{$\gc(t)$} is contained in the Clifford torus which is the common boundary of the two solid tori. The magnetic geodesic spirals around the two Reeb orbits \textcolor{orange}{$\{z_2=0\}$} and \textcolor{violet}{$\{z_1=0\}$} which are the souls of the solid tori. (Picture made by Ana Chavez Caliz.)}	\label{Hopf half}
\end{figure}	
\end{rem}
\subsection{Restriction of $(\Sd,g,\d\alpha)$ around a Hopf fibre}\label{ss:manehopf}
We can use the Hopf map $\pi\colon \Sd\to\SS^2(\tfrac12)$ also to find suitable primitives of $\d\alpha$ centered around the Hopf fibre of a point $q_0\in\Sd$ that give better bounds for the Mañé critical value of the magnetic system restricted to the sets
\[
B_{q_0,k}:=\{q_1\in\Sd\mid |\langle q_0,q_1\rangle|>\sqrt{1-8k}\},\quad k\in(0,\tfrac18).
\]
Up to magnetomorphisms, we can take $q_0=(1,0)$. On $\SS^2(\tfrac12)$, we consider polar coordinates $(r,\phi)\in[0,\tfrac\pi2]\times \SS^1$ around the north pole $\nu_+:=\pi(q_0)$. In these coordinates, we have
\[
g=\d r^2+\tfrac{\sin^2(2r)}{4}\d \phi^2,\qquad \sigma=\tfrac12\sin(2r)\d r\wedge\d\phi.
\]
We consider the primitive
\[
\beta:=\tfrac12\sin^2 r\d\phi
\]
of $\sigma$ in $\SS^2(\tfrac12)\setminus\{\nu_-\}$, the complement of the south pole $\nu_-$. For every $(r,\phi)$, we have
\begin{equation}\label{e:normbeta}
|\beta_{(r,\phi)}|=\tfrac12\sin^2r\frac{2}{\sin(2r)}=\tfrac{1}{2}\tan r.
\end{equation}

For every $\lambda\in[0,1]$, the one-form
\[
\alpha_\lambda:=\lambda\alpha+(1-\lambda)\pi^*\beta
\]
is a primitive of $\d\alpha$ on $B_{q_0,0}=\pi^{-1}(\SS^2(\tfrac12)\setminus \{\nu_-\})$ since \[
\d\pi^*\beta=\pi^*\d\beta=\pi^*\sigma=\d\alpha.
\]
Let $q\in B_{q_0,0}$ be arbitrary and let $(r,\phi)$ be the polar coordinates of $\pi(q)$. We compute
\begin{align*}
|(\alpha_\lambda)_q|^2=|\lambda\alpha_q+(1-\lambda)(\pi^*\beta)_q|^2&=\lambda^2|\alpha_q|^2+2\lambda(1-\lambda)g_q(\alpha_q,(\pi^*\beta)_q)+(1-\lambda)^2|(\pi^*\beta)_q|^2\\
&=\lambda^2\tfrac14+2\lambda(1-\lambda)(\pi^*\beta)_q(\tfrac14R_q)+(1-\lambda)^2|\beta_{\pi(q)}|^2\\
&=\tfrac14\lambda^2+(1-\lambda)^2|\beta_{\pi(q)}|^2,\\
&=\tfrac14\Big(\lambda^2+(1-\lambda)^2\tan^2r\Big)
\end{align*}
where in the last three steps we respectively used that $\tfrac14R$ is the dual of $\alpha$, that 
\[
\pi^*\beta(R)=\beta(\d\pi R)=0
\]
since the Reeb flow $R$ is tangent to the Hopf fibres, and that \eqref{e:normbeta} holds.

For every $k\in(0,\tfrac18)$, let $r_k:=\arccos (\sqrt{1-8k})\in(0,\tfrac\pi2)$ and compute
\begin{equation}\label{e:supra}
\begin{aligned}\sup_{q\in B_{q_0,k}}\tfrac12|(\alpha_\lambda)_q|^2&=\sup_{r\in[0,r_k)}\tfrac18\Big(\lambda^2+(1-\lambda)^2\tan^2r\Big)\\
&=\tfrac18\Big(\lambda^2+(1-\lambda)^2\tan^2r_k\Big)\\
	&=\tfrac18\frac{1}{\cos^2r_k}\Big(\lambda^2-2\lambda\sin^2r_k+\sin^2r_k\Big).
\end{aligned}
\end{equation}
To minimize the last expression in \eqref{e:supra}, we choose $\lambda:=\sin^2r_k\in(0,1)$ and get
\[
\tfrac18\frac{1}{\cos^2r_k}\Big(\lambda^2-2\lambda\sin^2r_k+\sin^2r_k\Big)=\tfrac18\frac{1}{\cos^2r_k}\Big(-\sin^4r_k+\sin^2r_k\Big)=\tfrac18\sin^2r_k=k,
\]
where the last equality stems from the definition of $r_k$.

This shows that the Mañé critical value of the system restricted to $B_{q_0,k}$ is less than or equal to $k$. More precisely, the function 
\[
F(q,v):=\sqrt{2k}|v|_q-(\alpha_\lambda)_q(v)
\]
yields a Finsler metric on $B_{q_0,k}$. The critical points of the length functional of $F$ are magnetic geodesics with energy $k$. It seems a non-trivial to find critical points of the length functional by variational methods since $B_{q_0,k}$ is not closed and $F$ degenerates on its boundary. In particular, one could ask about the existence of length minimizers connecting $q_0$ and an arbitrary point of $B_{q_0,k}$.
\section{Magnetomorphisms and totally magnetic submanifolds}\label{section totally magnetic submnfds}
A classical topic in Riemannian geometry going back to the work of E.~Cartan and M.~Berger is the classification of totally geodesic submanifolds $(N,g_N)$ of a given Riemannian manifold $(M,g_M)$, up to isometry.  So the aim of this section is to study the analog problem in the setting of magnetic systems.
\begin{rem}
The arguments in this section also holds in the category of Riemannian Hilbert manifolds.
\end{rem}
\subsection{The group of magnetomorphisms}
First, let us fix the notion of isomorphism between magnetic systems, as sketched in the introduction. 
\begin{defn}\label{magnetic morphism and iso}
	Let $(M_1,g_1,\sigma_1)$ and $(M_2,g_2,\sigma_2)$ be magnetic systems. A \emph{magnetomorphism} $\Phi\colon M_1\to M_2$ is a diffeomorphim such that
	\[
	\Phi^*g_2=g_1,\quad \Phi^*\sigma_2=\sigma_1.
	\]
	The \emph{group of magnetomorphisms} of a magnetic system $(M,g,\sigma)$ is defined as \[
	\mathrm{Mag}(M,g,\sigma):=\left\{\Phi:M\to M\mid \Phi^*g=g,\ \Phi^*\sigma=\sigma\right\}.
	\]
\end{defn}
\begin{rem}
	If $\Phi\colon(M_1,g_1,\sigma_1)\to(M_2,g_2,\sigma_2)$ is a magnetomorphism and $Y_1$ and $Y_2$ represent the Lorentz forces of the two magnetic systems, then
	\begin{equation}\label{e:Ymag}
		\d\Phi\cdot Y^1=Y^2_\Phi\cdot\d\Phi.
	\end{equation}
\end{rem}
\begin{prop}\label{p:magneto}
	If $\Phi\colon(M_1,g_1,\sigma_1)\to(M_2,g_2,\sigma_2)$ is a magnetomorphism, then $\Phi$ sends magnetic geodesics of $(M_1,g_1,\sigma_1)$ to magnetic geodesics of $(M_2,g_2,\sigma_2)$ with the same kinetic energy. 
\end{prop}
\begin{proof}
	Let $\gamma$ be a magnetic geodesic for $(M_1,g_1,\sigma_1)$, that is, 
	\begin{equation}\label{e:mag1}
		\nabla^1_{\dot\gamma}\dot\gamma=Y^1_\gamma\dot\gamma.
	\end{equation}
	The curve $\delta:=\Phi(\gamma)$ has the same kinetic energy as $\gamma$ since $\Phi$ is an isometry. Moreover,
	\[
	\nabla^2_{\dot\delta}\dot\delta=\nabla^2_{\d\Phi\dot\gamma}\d\Phi\dot\gamma=\d\Phi\cdot\nabla^1_{\dot\gamma}\dot\gamma=\d\Phi\cdot Y^1_\gamma\dot\gamma=Y^2_{\Phi(\gamma)}\d\Phi\dot\gamma=Y^2_{\delta}\dot\delta,
	\]
	where in the second step we used that $\Phi$ is an isometry, in the third step we used \eqref{e:mag1}, and in the fourth step we used \eqref{e:Ymag}. Thus $\delta$ is a magnetic geodesic for $(M_2,g_2,\sigma_2)$.
\end{proof}
We can prove the first part of Corollary \ref{c:symmetry} characterizing the magnetomorphisms of $(\S,g,\d\alpha)$.
\subsection{Proof of Corollary \ref{c:symmetry} (First part)}
Let $\Phi$ be a magnetomorphism of $(\S,g,\d\alpha)$. By definition, $\Phi$ is an isometry of $g$. By a classical result \cite{Pet}, we have
\[
\Phi\in O(2(n+1)).
\]
On the other hand, since $\Phi$ preserves $\d\alpha$, it preserves the Reeb direction, which means that 
\[
\Phi \mathrm iz=\pm \mathrm i\Phi z, \qquad \forall\,z\in \S.
\]
Since $\Phi$ is an isometry, we deduce that 
\begin{equation}\label{e:phic}
\Phi((\C z)^\perp)\subset (\C \Phi z)^\perp.
\end{equation}
Moreover, by \eqref{e:Ymag} and \eqref{e:YS}, we get
\[
\Phi\mathrm i P_{(\C z)^\perp}=\mathrm iP_{(\C \Phi z)^\perp}\Phi.
\]
Restricting this equation to $(\C z)^\perp$ and using \eqref{e:phic}, we get $\Phi\mathrm i=\mathrm i\Phi$ on the complex hyperplane $(\C z)^\perp$. Since $z\in \S$ was arbitrary, we conclude that $\Phi\mathrm i=\mathrm i\Phi$ on the whole $\C^{n+1}$, which shows that $\Phi\in U(n+1)$. From this we conclude that 
\[
U(n+1)=\mathrm{Mag}(\S,g,\d\alpha).
\]
We now show that the tangent lift $(\Phi,\d\Phi)\colon T\S\to T\S$ of the $U(n+1)$ action is Hamiltonian with respect to $\omega_{\d\alpha}=\d\lambda-\pi_{TM}^*\d\alpha$ and we compute its momentum map. First, since $\Phi$ is an isometry, the tangent lift is conjugated by the metric isomorphism to the cotangent lift $(\Phi,(\d\Phi^{-1})^*)\colon T^*\S\to T^*\S$. Since we know that that the cotangent lift preserves the canonical one-form we deduce that
\begin{equation}\label{e:philambda}
(\Phi,\d\Phi)^*\lambda=\lambda.
\end{equation}
Moreover, since $\Phi\in U(n+1)$, it follows from the definition of $\alpha$ that $\Phi^*\alpha=\alpha$ and therefore
\begin{equation}\label{e:phialpha}
(\Phi,\d\Phi)^*\pi_{TM}^*\alpha=(\pi_{TM}\circ(\Phi,\d\Phi))^*\alpha=(\Phi\circ\pi_{TM})^*\alpha=\pi_{TM}^*\Phi^*\alpha=\pi_{TM}^*\alpha.
\end{equation}
Consider now a one-parameter subgroup of $U(n+1)$ generated by $A\in\mathfrak u(n+1)$ and denote by $X_A$ the vector field generating the corresponding action on $T\S$. By \eqref{e:philambda} and \eqref{e:phialpha}, we get
\[
0=\mathcal L_{X_A}(\lambda-\pi^*\alpha)=\iota_{X_A}\omega_{\d\alpha}+\d\Big(\lambda(X_A)-\pi^*\alpha(X_A)\Big).
\]
Thus, $X_A$ is Hamiltonian with Hamiltonian function
\[
H(z,v)=\lambda(X_A)_{(z,v)}-\pi^*\alpha(X_A)_{(z,v)}=g_z(Az,v)-\alpha_z(Az),
\]
as we wanted to show.
\subsection{A characterization of totally magnetic submanifolds}
We recall the definition of totally magnetic submanifolds of a magnetic system $(M,g_M,\sigma_M)$ given in the introduction.
\begin{defn}\label{Definition totally magnetic submanifold}
	Let $(M,g_M, \s_M)$ be a magnetic system. An embedded submanifold $N\subset M$ is called totally magnetic if for all magnetic geodesics $\gamma\colon I\to M$ such that $\gamma(0)\in N$ and $\dot\gamma(0)\in T_{\gamma(0)}N$ there exist $\varepsilon>0$ such that $\gamma((-\varepsilon,\varepsilon))\subset N$.
\end{defn}
\begin{rem}
Magnetomorphisms send totally magnetic submanifolds to totally magnetic submanifolds.
\end{rem}
	\begin{rem}
	If $N$ is a totally magnetic submanifold of $M$ which is also a closed subset, then if $\gamma\colon I\to M$ is a magnetic geodesic tangent to $N$, we have $\gamma(I)\subset N$.
	\end{rem}
	\begin{rem}
	When $\sigma=0$, totally magnetic submanifolds recover the notion of totally geodesics submanifolds.
	\end{rem}
	We now prove Theorem \ref{t:totmag}, which turns the local definition of total magnetic submanifolds into an infinitesimal characterization. In order to do so, we define the following objects for an arbitrary embedded submanifold $N$ of $M$. First, let $\iota\colon N\to M$ be the inclusion map and let $\mathrm{P}_{TN}\colon TM|_N\to TN$ be the orthogonal projection. We define the magnetic system $(N,g_N,\sigma_N)$, where 
	\begin{equation}\label{e:gsigma}
	g_N:=\iota^*g_M,\qquad \sigma_N:=\iota^*\sigma_M.
	\end{equation}
	We denote by $\nabla^N$ and $\nabla^M$ the Levi-Civita connections of $g_N$ and $g_M$ and by $Y^N$ and $Y^M$ the Lorentz forces associated with the magnetic systems on $M$ and on $N$. Definition \eqref{e:gsigma} implies that
	\begin{equation}\label{e:ymyn}
	\nabla^N_v(X|_N)=\mathrm P_{TN}\nabla^M_vX,\quad  Y^Nv=\mathrm P_{TN}Y^Mv,
	\end{equation}
	for every $v\in TN$ and every vector field $X$ on $M$.
	
A submanifold \( N \subset M \) is totally geodesic if and only if its \emph{second fundamental form}
\[
\mathrm{II}_q\colon T_qN \to T_qN^{\perp},\quad \mathrm{II}_q(v) := \nabla^M_v V - \nabla^N_v V, \quad \forall\, q \in N
\]  
vanishes identically, that is, $\mathrm{II}_q(v) = 0$ for all $(q, v) \in TN$.

As we will see in Theorem \ref{t:totmag}, totally magnetic submanifolds are exactly those totally geodesic submanifolds for which the intrinsic and extrinsic Lorentz force coincide, that is, $Y_N=Y_M$ on $TN$.
\begin{proof}[Proof of Theorem \ref{t:totmag}]
\mbox{}
	
\underline{$(1)\Rightarrow (2)$:}
	Let $\delta\colon I\to N$ be a magnetic geodesic of $(N,g_N,\sigma_N)$. Let $t_0\in I$ and let $\gc$ be a magnetic geodesic in $\left(M,g_M,\sigma_M\right)$ such that $\gc(t_0)=\delta(t_0)$ and $\dot\gc(t_0)=\dot\delta(t_0)$. By (1), there exists $\varepsilon>0$ such that $\gamma((t_0-\varepsilon,t_0\varepsilon))\subset N$. In particular, $\dot\gamma(t)\in T_{\gamma(t)}N$ for all $t\in(t_0-\varepsilon,t_0+\varepsilon)$ and therefore
	\[
	\nabla_{\dot\gamma}^M\dot\gamma=Y^M_\gamma\dot\gamma\quad\Longrightarrow\quad \mathrm P_{TN}\nabla_{\dot\gamma}^M\dot\gamma=\mathrm P_{TN}Y^M_\gamma\dot\gamma\quad\Longleftrightarrow\quad \nabla_{\dot\gamma}^N\dot\gamma=Y^N_\gamma\dot\gamma
	\]
	Therefore, $\gamma$ is a magnetic geodesic for $(N,g_N,\sigma_N)$ and we conclude from the uniqueness of solutions to ordinary differential equations that $\delta|_{(t_0-\varepsilon,t_0+\varepsilon)}=\gamma|_{(t_0-\varepsilon,t_0+\varepsilon)}$. Since $t_0$ was arbitrary, we conclude that $\delta$ is also a magnetic geodesic for $(M,g_M,\sigma_M)$.
	
	\underline{$(2)\Rightarrow (1)$:} Let $\gamma\colon I\to M$ be a magnetic geodesic for $(M,g_M,\sigma_M)$ such that $(\gamma(0),\dot\gamma(0))\in TN$. Let $\delta\colon(-\varepsilon,\varepsilon)\to N$ be a magnetic geodesic for $(N,g_N,\sigma_N)$ with the same initial conditions as $\gamma$. By (2), $\delta$ is also a magnetic geodesic for $(M,g_M,\sigma_M)$. By uniqueness of solutions to ordinary differential equations, we conclude that $\delta=\gamma|_{(-\varepsilon,\varepsilon)}$. In particular, $\gamma((-\varepsilon,\varepsilon))\subset N$.
	 
	\underline{$(2)\Leftrightarrow (3a)$:}
	If $\delta$ is any curve on $N$, we have
	\[
	\nabla^N_{\dot\delta}\dot\delta-Y^N_\delta\dot\delta=\nabla^M_{\dot\delta}\dot\delta-Y^M_\delta\dot\delta+\II_\delta(\dot\delta)+Y^M_\delta\dot\delta-Y^N_\delta\dot\delta.
	\]
	If (2) holds and $\delta$ is a magnetic geodesic for $(N,g_N,\sigma_N)$, then $0=0+\II_\delta(\dot\delta)+Y^M_\delta\dot\delta-Y^N_\delta\dot\delta$. Statement (3a) follows from the quadratic nature of \( \mathrm{II} \) in \( v \), the linearity of the Lorentz force difference in \( v \), and the arbitrariness of the initial conditions of \( \delta \).  Vice versa, if (3a) holds, the equation above shows that
	\[
	\nabla^N_{\dot\delta}\dot\delta-Y^N_\delta\dot\delta=0\quad\iff\quad \nabla^M_{\dot\delta}\delta-Y^M_\delta\dot\delta=0.
	\]
	\begin{comment}
\underline{$(3)\Rightarrow (4)$:} If $u\in T_qN^\perp$ is arbitrary, then for all $\lambda\in\RR$, we get
\[
0=g_q(\II^{\mathrm{mag}}_q(\lambda v),u)=\lambda^2g_q(\II_q(v),u)+\lambda g_q(Y^M_qv-Y^N_qv,u),
\]
which is a polynomial of degree two in $\lambda$. Since a real polynomial is zero if and only if all its coefficients are zero, we conclude that for all $u\in T_qN^\perp$, $g_q(\II_q(v),u)=0$ and $g_q(Y^M_qv-Y^N_qv,u)=0$. Hence, $\II_q=0$ and $Y^M|_{TN}=Y^N$ since $u$ was arbitrary.

\underline{$(4)\Rightarrow (3)$:} From the definition, we see that $\II_q=0$ and $Y^M|_{TN}=Y^N$ imply $\II^{\mathrm{mag}}=0$.
	\end{comment}

\underline{$(3a)\Leftrightarrow (3b)$:} From equation \eqref{e:ymyn}, we see that $Y^M|_{TN}-Y^N=\mathrm{P}_{TN^\perp}Y^M|_{TN}$ and the equivalence follows.

\underline{$(3b)\Leftrightarrow (3c)$:} If $u,v\in T_qM$, then
\[
(\sigma_M)_q(u,v)=g_q(Y^M_qu,v).
\]
Therefore, the left-hand side is zero for all $u\in T_qN$ and $v\in T_qN^\perp$ if and only if the right-hand side is zero for all $u\in T_qN$ and $v\in T_qN^\perp$. Finally, the right-hand side is zero for all $u\in T_qN$ and $v\in T_qN^\perp$ if and only if $Y^M_qu\in (T_qN^\perp)^\perp=T_qN$ for all $u\in T_qN$.
\end{proof}
\begin{rem}
If we fix the kinetic energy $k$ (or the strength $s$), one could look at the larger class of submanifolds $N$ which are totally magnetic at energy $k$, where in Definition \ref{Definition totally magnetic submanifold}, we require that the magnetic geodesic $\gamma$ of energy $k$. In this case, suitable versions of (1), (2) as above are still equivalent. Notice that a magnetic geodesic with kinetic energy $k$ is a one-dimensional totally magnetic submanifold at energy $k$. With the stronger notion of being totally magnetic that we use in this article, a magnetic geodesic is totally magnetic if and only if it is a standard geodesic.
\end{rem}

\subsection{Proof of Corollary \ref{c:symmetry} (Second Part)}
We can now use Theorem \ref{t:totmag} to classify \emph{closed, connected} totally magnetic submanifolds $N$ of $(\S,g,\d\alpha)$ with positive dimension and finish the proof of Corollary \ref{c:symmetry}. 

By Theorem \ref{t:totmag}, $N$ is totally geodesic in $(\S,g)$. By the classification of closed, connected totally geodesic submanifolds with positive dimension of spheres of constant curvature \cite[§3,Thm]{Wolf63}, $N$ is a $j$-dimensional Euclidean sphere for some $1\leq j\leq 2n+1$. This means that there exists a real vector subspace $E\subset \CC^{n+1}$ with $\dim_{\RR}E=j+1$ such that 
	\begin{equation}\label{equation 1 for classification }
		N= E\cap \S\subseteq \CC^{n+1}.
	\end{equation}
Our aim is to show that $E$ is a complex vector subspace. Since $N$ is totally magnetic, by condition (3a) in Theorem \ref{t:totmag} we see that for every $z\in N$
\begin{equation}\label{e:YN}
\mathrm i P_{(\C z)^\perp}u=Y^N_zu\in T_zN,\quad \forall\,u\in T_zN.
\end{equation}
We distinguish two cases. In the first case $\dim_{\RR}E=2$. If $z\in N\subset E$, then 
\begin{equation}\label{e:zu}
E=\RR z\oplus \RR u,
\end{equation}
where $u\in z^\perp =T_zN$. It follows that there exists $c\in\RR$ and $v\in (\C z)^\perp$ such that $u=c\mathrm iz+v$. By applying \eqref{e:YN}, we get that 
\[
\mathrm iv=\mathrm i P_{(\C z)^\perp}u\in T_zN.
\] 
From \eqref{e:zu} and the expression for $u$, this forces $v=0$ and hence $c\neq0$ since $\dim_{\RR}E=2$. We conclude that $E=\RR z\oplus \RR(\mathrm iz)=\C z$ is complex.

In the second case, $\dim_{\RR}E\geq 3$. If $z\in N$, then $E\cap (\C z)^\perp\subset T_zN$. Therefore, by \eqref{e:YN}
\[
\mathrm i u\in T_zN\subset E,\qquad \forall\, u\in E\cap (\C z)^\perp.
\]
Thus to show that $E$ is complex, it is enough to prove that for every $u\in N$, there is $z\in N$ such that $u\in E\cap (\C z)^\perp$. Indeed, since $\dim_{\RR}E\geq 3$, the vector space $E\cap (\C u)^\perp$ is non-zero. Thus take $z$ of unit norm contained in $E\cap(\C u)^\perp$. On the one hand, $z\in N$. On the other hand, since $z\in(\C u)^\perp$, we deduce that $u\in (\C z)^\perp$. Hence, $z\in N$ and $u\in E\cap (\C z)^\perp$ as desired.

Now that we have shown that $N=E\cap \S$ where $E$ is complex subspace, let us call $j$ the complex dimension of $E$ and show that $(N,g_N,\sigma_N)$ is magnetomorphic to $(\SS^{2j+1},g,\d\alpha)$. By the first part of Corollary \ref{c:symmetry} there is a magnetomorphism $\Phi\in U(n+1)$ of $\S$ such that 
\[
N=\Phi\Big((\C^j\times 0)\cap \S\Big)=\Phi\Big(\SS^{2j+1}\times \{0\}\Big).
\]
The map $\C^j\to\C^j\times\{0\}$, $z\mapsto (z,0)$ restricts to a diffeomorphism $\Psi\colon \SS^{2j+1}\to \SS^{2j+1}\times \{0\}$ such that 
\[
\Psi^*\Phi^*(g_N,\sigma_N)=\Psi^*(g_{\SS^{2j+1}\times \{0\}},\sigma_{\SS^{2j+1}\times \{0\}})=(g,\d\alpha),
\] 
where $g$ and $\alpha$ are the Euclidean metric and standard contact form on $\SS^{2j+1}$.
	\bibliographystyle{abbrv}
	\bibliography{ref}
	
\end{document}